\documentclass[a4paper,12pt]{article}

\usepackage[T2A]{fontenc}

\usepackage[english]{babel}

\usepackage{graphicx}
\usepackage{amsthm, amsmath, amssymb, amsbsy, amsfonts, dsfont}
\usepackage{srcltx, color}

\def\cF{\mathcal{F}}

\def\Om{\Omega}

\def\g{\gamma}

\def\l{\lambda}
\def\p{\partial}

\def\a{\alpha}
\def\b{\beta}

\def\d{\delta}

\def\z{\zeta}
\def\vp{\varphi}

\def\iu{\mathrm{i}}

\def\cL{\mathcal{L}}

\def\cA{\mathcal{A}}

\def\cH{\mathcal{H}}

\def\cI{\mathcal{I}}

\def\cQ{\mathcal{Q}}
\def\cL{\mathcal{L}}

\def\ru{\mathrm{u}}

\def\fa{\mathfrak{a}}
\def\fl{\mathfrak{l}}
\def\fr{\mathfrak{r}}

\DeclareMathOperator{\spec}{\sigma}
\DeclareMathOperator{\discspec}{\sigma_{disc}}
\DeclareMathOperator{\essspec}{\sigma_{ess}}

\DeclareMathOperator{\RE}{Re}

\numberwithin{equation}{section}

\newtheorem{theorem}{Theorem}[section]
\newtheorem{lemma}{Lemma}[section]

\begin{document}

\allowdisplaybreaks


\title{On spectra of convolution operators with potentials}

\date{\empty}

\author{D.I. Borisov$^{1,2,3}$, A.L. Piatnitski$^{4,5}$, E.A. Zhizhina$^4$}

\maketitle
{\small
 \begin{quote}
1) Institute of Mathematics, Ufa Federal Research Center, Russian Academy of Sciences, Chernyshevsky str. 112, Ufa, Russia, 450008
\\
2) Bashkir State University, Zaki Validi str. 32, Ufa, Russia, 450076
\\
3) University of Hradec Kr\'alov\'e
62, Rokitansk\'eho, Hradec Kr\'alov\'e 50003, Czech Republic
\\
4) Institute for Information Transmission Problem (Kharkevich Institute) of RAS, Bolshoy Karetny per. 19, build.1, Moscow, Russia, 127051
\\
5) The Arctic University of Norway, campus Narvik, PO Box 385, N-8505 Narvik, Norway
\\
Emails: borisovdi@yandex.ru, apiatnitski@gmail.com, ejj@iitp.ru
\end{quote}}

{\small \begin{quote}
 \noindent{\bf Abstract.}  This paper focuses on the spectral properties of a bounded self-adjoint operator in $L_2(\mathds R^d)$ being
 the sum of a convolution operator with an integrable convolution kernel and an operator of  multiplication by a continuous potential converging to zero at infinity.
 We study both the essential and the discrete spectra of this operator. It is shown that
 the essential spectrum of the sum is the union of the essential spectrum of the convolution operator and the image of the potential.
 We then provide a number of sufficient conditions for the existence of  discrete spectrum and obtain lower and upper bounds for the
 number of discrete eigenvalues.  Special attention is paid to the case of operators possessing countably many points of the discrete spectrum.
We also compare the spectral properties of the operators considered in this work with those of classical Schr\"odinger operators.

 \medskip

 \noindent{\bf Keywords:} convolution operator, potential,  essential spectrum, discrete spectrum, minimax principle.

 \medskip

 \noindent{\bf Mathematics Subject Classification:}
 \end{quote}}

\section{Introduction}

In this work we study the spectral properties of a non-local self-adjoint operator $\cal L$ in $L_2(\mathds{R}^d)$ of the form
\begin{equation}\label{intro-1}
\cL u = \cA u+ V u, \qquad  (\cA u)(x): = \int\limits_{\mathds{R}^d} a(x-y) u(y) \, dy,
\end{equation}
where $\cA$ is a convolution operator with an integrable kernel $a(z)$, and  $V$ is a potential being a bounded continuous real-valued function that tends to zero at infinity.
Our goal is to characterize the structure of the essential and discrete spectrum of this operator.

 In recent years there is a  growing attention
to non-local convolution type operators with integrable
kernels. This is stimulated by a number of interesting and non-trivial mathematical problems appearing in
the theory of such operators on the one hand, and by various important applications of this theory on the other hand.
Among the applied fields    in which  zero order convolution type operators are of essential importance we mention population
dynamics,  ecological problems and  porous media theory.
In particular, in the population dynamics models the operators defined in (\ref{intro-1}) with a non-negative function $a(\cdot)$ can be  used to analyse the spread of infections or the growth of biological populations of plants or animals.

A rigorous mathematical theory of population dynamics relies on the so called contact model in continuum,  see e.g.  \cite{KKP, PiroZh, MW}.
This model deals with  a birth and death process that describes the evolution of stochastic interacting infinite-particle systems in terms of birth and death rates. The function $a(\cdot)$ is called the dispersal kernel, it defines the distribution of a position of  a newly born particle
in the  configuration.

The evolution of the first correlation function  denoted by $u(x,t)$ and being the density of a population is described  by the following  Cauchy problem:
\begin{equation}\label{intro-4}
\frac{\partial u}{\partial t} \ = \ {\cal L}u-\langle a\rangle u , \quad u = u (t, x), \quad x \in \mathds{R}^d, \quad t \geqslant 0, \qquad u(0,x) = u_0 (x) \geqslant 0,
\end{equation}
where $\langle a\rangle=\int\limits_{\mathds{R}^d}a(z)\,dz$.
The potential $V(x)$ in (\ref{intro-4}) is a real-valued function defined as the difference between the birth and death rates at point $x \in \mathds{R}^d$.  In spatially inhomogeneous environments the birth and death rates are functions of the position in the space
and thus the potential $V(x)$ need not be equal to a constant.   It is assumed that at infinity the birth and death rates coincide so that $V(x)$ tends to zero as $|x|\to\infty$.

This gives rise to an interesting mathematical question that reads: find a class of potentials $V(\cdot)$ and dispersal kernels $a(\cdot)$
for which the operator $\mathcal{L}-\langle a\rangle$ has a positive spectrum and thus the density of population shows an exponential growth
everywhere in the space.
The problem of the existence of a positive eigenvalue
has been discussed in
\cite{BC, KPMZh, KMV}, and the structure of the principal eigenfunction has been investigated in \cite{KPMZh18}.

It is known, see, for instance, \cite[Theorem 19.1]{Bhat_Rao},  that in the region $x\lesssim \sqrt{t}$ the large time asymptotics of the fundamental solution of the
equation
\begin{equation}\label{intro-6}
\partial_t v(x,t)=\int_{\mathds R^d}a(x-y)\big(v(y,t)-v(x,t)\big)\,dy
\end{equation}
 coincides with that of the heat kernel of the operator $\mathrm{div}\big(\hat a\nabla\big)$ with
 $$
 \hat a_{ij}=\int_{\mathds R^d} z_iz_ja(z)dz.
$$
Therefore, it is natural to consider the operator on the right-hand side of (\ref{intro-6}) as an approximation of the Laplacian
and to call $\mathcal{L}$ a non-local Schr\"odinger operator.

The operator $\cL$ defined in (\ref{intro-1}) can be regarded as a perturbation of the convolution operator $\cA$ by the potential $V$ or vice versa, a perturbation of the multiplication operator by the convolution one. From this point of view, there is a clear analogy with spectral properties of the Schr\"odinger operators perturbed by potentials or other lower order perturbations. There is a vast literature and hundreds of works devoted to such operators. Not trying to mention all of them, we just cite few classical works \cite{Simon1}, \cite{Simon2}, \cite{Klaus}, \cite{Klaus2} and a recent book \cite{ExnKov}. However, there exists a fundamental difference between the classical Schr\"odinger operator and the non-local Schr\"odinger operator of the form (\ref{intro-1}). The potential term in the classical Schr\"odinger operator is a relatively compact perturbation of the Laplacian, while
 the terms of the operator $\mathcal{L}$ in (\ref{intro-1}) are equipollent.
As our main results show, this fact makes
  the spectral pictures for such operators and for classical Schr\"odinger operators rather different.

In the present paper    two conditions are imposed  on the kernel $a(\cdot)$. Namely, we assume that
$a(-z)=\overline{a(z)}$ for all $z\in\mathds{R}^d$, and $a\in L_1(\mathds{R}^d)$. The first condition makes
the operator $\mathcal{L}$ symmetric, while the second one ensures that $A$ is bounded in $L_2(\mathds{R}^d)$.
The function $V$  is  real, continuous, and vanishes at infinity.
Observe that the operator $\cA$ is unitary equivalent to the operator of multiplication by the Fourier image of the
function $a(\cdot)$ and that this Fourier image is a continuous function that vanishes at infinity.

Our first result characterizes the  essential spectrum of operator $\cL$ in $L_2(\mathds{R}^d)$. We show that
$\essspec(\cL)$ is the union of the spectra of $\cA$ and of the multiplication operator $u\mapsto Vu$.

Then we provide a number of sufficient conditions for the existence of the discrete spectrum and obtain several upper and lower bounds for
the number of points of the discrete spectrum.
The lower bounds rely on the detailed analysis of the  convolution operator and the minimax principle.
In order to prove an upper bound, we use a modification of the Birman-Schwinger principle
 adapted to the non-local operators studied here.

 We also pay a special attention to the cases, when the operator $\cL$ possesses infinitely many discrete eigenvalues accumulating to the edges of the essential spectrum. We provide various sufficient conditions guaranteeing such
a behaviour of the spectrum. In particular, these conditions show that the class of non-local Schr\"odinger operators having infinitely many points of the discrete spectrum is  rather wide in contrast with classical differential operators with lower order perturbations.

\section{Problem setup
and main results}

Let
$V=V(x)$ and $a=a(x)$ be given measurable complex-valued functions defined on $\mathds{R}^d$. We assume that the function $a$ belongs to $L_1(\mathds{R}^d)$
and satisfies the identity
\begin{equation}\label{2.1}
a(-x)=\overline{a(x)}.
\end{equation}
By $\cF$ we denote a Fourier transform on $L_1(\mathds{R}^d)$ defined
by the formula
\begin{equation*}
\cF[u](x):=
\int\limits_{\mathds{R}^d} u(\xi)e^{-\iu x\cdot \xi}\,d\xi.
\end{equation*}
The same symbol stands for the Fourier transform extended to $L_2(\mathds{R}^d)$.
We then assume that the function $V$ is an image of some function $\hat{V}\in L_1(\mathds{R}^d)$ satisfying also condition (\ref{2.1}), that is,
\begin{equation*}
V=\cF[\hat{V}],\qquad \hat{V}(-x)=\overline{\hat{V}(x)}.
\end{equation*}
We also denote
\begin{equation*}
\hat{a}(\xi):=\cF[a](\xi).
\end{equation*}

The main object of our study is an operator in $L_2(\mathds{R}^d)$ defined by the formula
\begin{equation*}
\cL:=\cL_{a\star} + \cL_V,\qquad (\cL_{a\star} u)(x):=
\int\limits_{\mathds{R}^d} a(x-y)u(y)\,dy,\qquad (\cL_V u)(x):=V(x)u(x).
\end{equation*}
We shall show, see Lemma~\ref{lmL}, that this operator is bounded in $L_2(\mathds{R}^d)$ and is self-adjoint. Our main aim is to describe the structure of the spectrum of this operator depending on the properties of the functions $a$ and $V$.

Observe that under the above assumptions on $a$ and $V$ the functions $\hat{a}$ and $V$ are real-valued, bounded, continuous and decaying at infinity. In view of these properties the following quantities are well defined:
\begin{equation}\label{2.7}
a_{\min}:=\inf\limits_{\mathds{R}^d} \hat{a},\qquad a_{\max}:=\sup\limits_{\mathds{R}^d} \hat{a},\qquad
V_{\min}:=\inf\limits_{\mathds{R}^d} V,\qquad V_{\max}:=\sup\limits_{\mathds{R}^d} V.
\end{equation}
It follows from the aforementioned properties of
$a$ and $V$ that
\begin{equation*}
 a_{\min}\leqslant 0\leqslant a_{\max},\qquad V_{\min}\leqslant 0\leqslant V_{\max}.
\end{equation*}

By $\essspec(\,\cdot\,)$ we denote an essential spectrum of an operator, while $\discspec(\,\cdot\,)$ stands for a discrete spectrum. The spectrum of an operator is denoted by $\spec(\,\cdot\,)$. Let $Q_r(x_0)$ be a cube in $\mathds{R}^d$ with a side $r$ centered at a point $x_0$.

Our first result describes the essential spectrum of the operator $\cL$.

\begin{theorem}\label{thES}
The essential spectrum of the operator $\cL$ coincides with the segment $[\mu_0,\mu_1]$, where $\mu_0:=\min\{a_{\min}, V_{\min}\}$, $\mu_1:=\max\{a_{\max}, V_{\max}\}$. The discrete spectrum of the operator $\cL$ can be located only in the semi-intervals $[a_{\min}+V_{\min},\mu_0)$ and $(\mu_1,a_{\max}+V_{\max}]$ and it can accumulate to the points $\mu_0$ and $\mu_1$ only.
\end{theorem}

The rest of our results describes the discrete spectrum of $\mathcal{L}$. First, we provide sufficient condition ensuring its existence.

\begin{theorem}\label{thDS1}
Let $x_0$ be a point of the global minimum of the function $V$, and assume that $V_{\min}\leqslant a_{\min}$. Assume furthermore that there exists $\d>0$ such that
\begin{equation}\label{2.8}
 \int\limits_{Q_2(0)} \prod\limits_{i=1}^{d} (1-|x_i|) \RE a(\d x)\,dx+
 \d^{-d} \int\limits_{Q_1(0)} \big(V(x_0+\d x)-V_{\min}\big)\,dx<0.
\end{equation}
Then the discrete spectrum of the operator $\cL$ in the semi-interval $[a_{\min}+V_{\min},\mu_0)$ is non-empty.
\end{theorem}

Once we know that the discrete spectrum is non-empty, we are interested in the number of the discrete eigenvalues. Various lower bounds for this number are provided in Theorems \ref{thDS5}--\ref{thDS2} below.

We fix some $r>0$ and denote
\begin{equation}\label{2.18}
a_n:=(2r)^{-d}\int\limits_{Q_{2r}(0)} a(x) e^{-\frac{\pi\iu}{r}n\cdot x} \,dx,\qquad n\in\mathds{Z}^d.
\end{equation}
Since $a\in L_1(Q_{2r}(0))$, all Fourier coefficients $a_n$ are well-defined. Employing identity (\ref{2.1}), it is straightforward to confirm that all constants $a_n$ are real-valued. We then introduce the following sets of indices:
\begin{equation*}
\mathds{J}_0:=\{n\in\mathds{Z}^d:\, a_{2n}<0\}
\end{equation*}
and assume that this set is not empty.

 Supposing  that $x_0$ is a point of the global minimum of the function $V(x)$, that is, $V_{\min}=V(x_0)$, we introduce
\begin{equation}\label{2.9}
V_n:=\int\limits_{Q_r(x_0)} (V(x+x_0)-V_{\min}) e^{\frac{2\pi\iu}{r}n\cdot x}\,dx,\qquad \overline{V_{-n}}=V_n,
\end{equation}
and, given a subset $\mathds{J}\subset \mathds{J}_0$, we denote
\begin{equation*}
\nu_\mathds{J}:=r^{-d}\sup\limits_{n\in \mathds{J}} \sum\limits_{m\in \mathds{J}} |V_{n-m}|.
\end{equation*}

\begin{theorem}\label{thDS5}
Assume that $V_{min}\leqslant a_{min}$, $x_0$ is a point of the global minimum of the function $V(x)$, and
 there exists a subset $\mathds{I}\subset \mathds{J}_0$ such that
\begin{equation}\label{2.16}
r^d\max\limits_{n\in \mathds{I}} a_{2n} + (2r)^d \sup\limits_{n\in\mathds{Z}^d\setminus (2\mathds{Z})^d} a_n + \nu_\mathds{I}<0.
\end{equation}
Then the operator $\cL$ possesses at least $\# \mathds{I}$ eigenvalues below $V_{\min}$, where $\# \mathds{I}$ is the total number of indices in the set $\mathds{I}$. The lowest eigenvalue $\l_{\min}$ of the operator $\cL$ satisfies the upper bound
\begin{equation}\label{2.17}
\l_{\min}\leqslant r^d \min\limits_{n\in\mathds{Z}^d} a_{2n} + (2r)^d \sup\limits_{n\in\mathds{Z}\setminus (2\mathds{Z})^d} a_n + r^{-d} \int\limits_{Q_r(x_0)} (V(x)-V_{\min})\,dx.
\end{equation}
\end{theorem}

In the next theorem we consider the case of a sufficiently smooth kernel $a$. Namely, given $N\in\mathds{N}$, we suppose that $a\in C^{2N+1}(Q_r(0))$ for some fixed $r>0$.
We
introduce a quadratic form
\begin{equation}\label{2.12}
\fa_N[\z]:=\sum\limits_{\substack{n,m\in \mathds{Z}_+^d \\ |m|, |n|\leqslant N}} (-1)^{|n|}
\p^{n+m} a(0)\z_m\overline{\z_n},\qquad \z:=(\z_n)_{ n\in\mathds{Z}_+^d,\, |n|\leqslant N},
\end{equation}
on $\mathds{C}^{M(N)}$, where $|n|=n_1+n_2+\ldots+n_d$ and
\begin{equation*}
M(N):=\#\big\{n\in\mathds{Z}_+^d:\, |n|\leqslant N\big\},\qquad n!=n_1!\cdot\ldots \cdot n_d!\quad\text{for}\quad n=(n_1,\ldots,n_d)\in\mathds{Z}_+^d.
\end{equation*}

Let $x_0$ be a point of the global minimum of $V$. We then let
\begin{equation}\label{2.13}
h_N(\d):=\max\limits_{\substack{n\in\mathds{Z}_+^d \\ |n|\leqslant 2N}} \bigg|
\int\limits_{Q_1(0)} \big(V(x_0+\d x)-V(x_0)\big)x^n\,dx\bigg|.
\end{equation}

\begin{theorem}\label{thDS4}
Let $x_0$ be a point of the global minimum of $V(\cdot)$, and $V_{\min}\leqslant a_{\min}$. Assume that $a\in C^{2N+1}(Q_r(0))$ with some $r>0$, the identity
\begin{equation}\label{2.14}
\lim\limits_{\d\to0} \frac{h_N(\d)}{\d^{2N+d}}=0
\end{equation}
holds and there exists a subspace $S$ in $\mathds{C}^{M(N)}$, on which
the form $\fa_N$ defined in (\ref{2.12}) is strictly negative.
Then the operator $\cL$ possesses at least $\dim S$ eigenvalues in the interval $[a_{\min}+V_{\min}, \mu_0)$.
\end{theorem}

As  it has been demonstrated
in Theorem \ref{thDS4}, sufficient conditions of the existence of a discrete spectrum of $\mathcal{L}$ can also be formulated in terms of the Taylor coefficients of $a(\cdot)$ about the origin  and the behaviour of $V$ in the vicinity of its minimum point.  Namely,  it suffices  to check the negative definiteness of the form $\fa_N$ on some subspace $S$ and the validity of (\ref{2.14}). In the next theorem we provide
a class of functions $a(\cdot)$ and $V(\cdot)$ for which these conditions hold.

\begin{theorem}\label{thDS6}
Let $x_0$ be a point of the global minimum of the function $V$, and assume that
\begin{itemize}
\item
 $V_{\min}\leqslant a_{\min}$.
 \item  The estimate
\begin{equation}\label{2.22}
V(x)-V(x_0)\leqslant C|x-x_0|^\a
\end{equation}
 holds
for all $x$ in a small neighbourhood of $x_0$, where $C$ and $\a$ are some positive constants independent of $x$.
\item There exists
a subset $\mathds{I}\subseteq \{n\in\mathds{Z}_+^d:\, |n|\leqslant N\}$ such that the derivatives of the function $a(\cdot)$ obey the conditions
\begin{gather}\label{2.19}
(-1)^{|n|} \p^{2n} a(0)<0,\qquad n\in\mathds{I},
\\[2mm]
|\p^{n+m}a(0)|\leqslant \b_{n,m} \sqrt{|\p^{2n}a(0)|} \sqrt{|\p^{2m}a(0)|},\qquad n,m\in\mathds{I},\quad n\ne m,\label{2.20}
\end{gather}
where $N<\frac{\a-d}{2}$ and $\b_{n,m}$, $n,\,m\in\mathds{I}$, are some non-negative numbers that satisfy at least one of the following
two conditions
\begin{equation}\label{2.21_bis}
 \b_1 :=\max\limits_{m\in\mathds{I}}\,
 \sum\limits_{\substack{n\in\mathds{I} \\  n\not=m}} \b_{n,m}< 1
\end{equation}
or
\begin{equation}\label{2.21}
 \b_2 :=\sum\limits_{\substack{n,\,m\in\mathds{I} \\  n\ne m}} \b_{n,m}^2 < \frac{(\#\mathds{I})^\frac12}{(\#\mathds{I})^\frac{1}{2}-1}.
\end{equation}
\end{itemize}
Then the operator $\cL$ possesses at least $\#\mathds{I}$ eigenvalues in the interval $[a_{\min}+V_{\min}, \mu_0)$.
\end{theorem}

  The following two theorems concern
  the operators $\cL$ possessing infinitely many discrete eigenvalues.  In the first of them we consider the case of a smooth convolution kernel.

\begin{theorem}\label{th_suf-infin}
Let $a\in C^\infty(Q_r(0))$, and assume that condition~(\ref{2.22}) holds with an arbitrary  $\a>0$.
Assume furthermore that  $V_{\min}\leqslant a_{\min}$ and there exist constants $\gamma>0$ and  $c_1,\,c_2>0$ and an infinite subset $\mathds{I}\subseteq \mathds{Z}_+^d$ such that
\begin{align}
\label{con1}
&(-1)^{|n|} \p^{2n}a(0)<0\ \hskip -4.4cm  &&\text{for all}\quad   n\in\mathds{I},\\
\label{con2}
&\big|\p^{2n}a(0)\big|\geqslant c_1((2n)!)^\gamma\ \hskip -4.4cm   &&\text{for all} \quad n\in\mathds{I},
\\
  \label{con3}
&\big|\p^{n}a(0)\big|\leqslant c_2(n!)^\gamma\  \hskip -4.4cm &&\text{for all}\quad  n\in\mathds{Z}_+^d.
\end{align}
Then the operator $\cL$ has infinitely many  eigenvalues below $\mu_0$.
\end{theorem}

Our next theorem describes the situation when the operator $\cL$ possesses infinitely many discrete eigenvalues  for  kernels that need not be smooth.

\begin{theorem}\label{thDS2}
Let $V_{\min}\leqslant a_{\min}$, $V(x)\equiv V_{\min}$ on some cube $Q_r(x_0)$ and assume that at least one of the following two conditions hold:

\begin{enumerate}
\item\label{thDS2it1}
The inequalities
\begin{equation}\label{2.10}
a_{\min}<0\quad\text{and}\quad a_{\max}=0
\end{equation}
are satisfied;

\item\label{thDS2it2} For all $n\in \mathds Z^d$ the quantities $a_n$ introduced in (\ref{2.18}) satisfy the inequalities
\begin{equation*}
a_n\leqslant 0
\end{equation*}
and there exists an infinite subsequence of indices in $\mathds{Z}^d$ such that on this subsequence the above inequalities are strict.
\end{enumerate}

\noindent Then the operator $\cL$ possesses countably many  eigenvalues in the semi-interval $[a_{\min}+V_{\min}, \mu_0)$, which accumulate to the point $\mu_0$.
\end{theorem}

In complement to the lower bounds for the number of discrete eigenvalues, we also provide an upper bound for this number in the following theorem.

\begin{theorem}\label{thDS3}
Let $\mu_0 = V_{\min}$,  and  assume that $V_{\min}\leqslant a_{\min}$ and that
\begin{equation*}
I_V:=\int\limits_{\mathds{R}^d}
\frac{
V_-(x)\,dx}{V_-(x)+V_{\min}}< \infty,\qquad I_a:=\frac{1}{(2\pi)^{d}} \int\limits_{\mathds{R}^d}
\frac{
\hat{a}_-(x)\,dx}{\hat{a}_-(x)+V_{\min}}< \infty,
\end{equation*}
where $V_-(x):=-\min\{0,V(x)\}$, $\hat{a}_-(x):=-\min\{0,\hat{a}(x)\}$. Then the number of the  eigenvalues of the operator $\cL$ below $\mu_0$ does not exceed $I_a I_V$.
\end{theorem}

\section{Discussion of main results}

In this section we discuss the principal aspects of our model and main results.  We begin with the fact mentioned already in the introduction: both terms $\cL_{a\star}$ and $\cL_V$ in the operator $\cL$ are bounded operators in $L_2(\mathds{R}^d)$ and none of them is relatively compact  with respect to the other. This is a fundamental difference in comparison with classical Schr\"odginger operators and it explains
 specific features of the spectra of operators considered here.

The first difference is already demonstrated by Theorem~\ref{thES}, which says that the essential spectrum of the operator $\cL$ is the union of those of  $\cL_{a\star}$ and $\cL_V$. For the classical Schr\"odinger operators with localized perturbations the essential spectrum is determined solely by the unperturbed operator, while in our case both the convolution and multiplication operators contribute to the essential spectrum.  The entire spectrum is a bounded set, which can be located only in the interval $[a_{\min}+V_{\min},a_{\max}+V_{\max}]$; this fact is due to the boundedness of operator $\cL$.

The next important question is about the existence of discrete spectrum. According to Theorem~\ref{thDS1}, it can be located only in semi-intervals $[a_{\min}+V_{\min},\mu_0)$ and $(\mu_1,a_{\max}+V_{\max}]$.  Our theorems deal with the eigenvalues located in the former semi-interval under the assumption that $V_{\min}\leqslant a_{\min}$. These results can be easily transferred to the case $a_{\min}\leqslant V_{\min}$ and also to the interval $(\mu_1,a_{\max}+V_{\max})$. Indeed, the opposite case $a_{\min}\leqslant V_{\min}$ can be   treated   by passing to a unitary equivalent operator
\begin{equation*}
\left(\frac{1}{(2\pi)^\frac{d}{2}}
\cF\right) \cL \left(\frac{1}{(2\pi)^\frac{d}{2}}
\cF\right)^{-1}=\cL_{\hat{a}} + \cL_{\hat{V}\star}.
\end{equation*}
In the latter operator, the functions $a$ and $V$ interchange their roles in the sense that the function $\hat{a}$ generates the multiplication operator
$\cL_{\hat{a}}$, while the function $\hat{V}$ produces the convolution operator $\cL_{\hat{V}\star}$.  In order to study the eigenvalues above the point $\mu_1=\max\{a_{\max}, V_{\max}\}$, we should simply replace the operator $\cL$ with $-\cL$.

Our first result on the discrete spectrum,  Theorem~\ref{thDS1}, gives a sufficient condition for the existence of $\sigma_{\rm disc}$. The first integral on the left-hand side of (\ref{2.8}) represents the contribution of the convolution kernel, while the second integral reflects that of the potential $V$. Since $V_{\min}$ is the global minimum of $V$, the second integral is obviously non-negative and, in order to make condition (\ref{2.8}) satisfied, the contribution of the convolution part should be negative. This condition is first of all aimed for the case of small $\d$. In this case the existence of the discrete spectrum depends on the local properties of both  the convolution kernel  in the vicinity of the origin and  the potential in the vicinity of its global minimum point.  If the kernel $a$ is continuous and the function $V$ satisfies the relation
$V(x)-V_{\min}=c_0|x-x_0|^{\a}(1+o(1)), \  c_0>0,\ \a>0$, as $|x-x_0|\to0$,
then
 condition (\ref{2.8}) can be rewritten in a simpler form:
\begin{equation}\label{2.8a}
\RE a(0)
 +\d^{-d+\a} c_0\int\limits_{Q_1(0)} |x|^\a\,dx<0.
\end{equation}

Of course,  the discrete spectrum can exist not only due to the local properties of the convolution kernel and the potential, but also due to their global structure. Such cases are also covered by Theorem~\ref{thDS1} once condition (\ref{2.8}) holds for some $\d>0$.

We also stress that in Theorem~\ref{thDS1} we do not suppose that the potential $V$ possesses a single point of the global minimum. If it has several such points, the theorem applies at each of them.

The above discussion shows that Theorem~\ref{thDS1} is quite universal.  It applies to rather general convolution kernels and potentials.

If the discrete spectrum of $\mathcal{L}$ is non-empty, it is natural to turn to estimating the number of discrete eigenvalues of $\mathcal{L}$.  Lower bounds for the total number of the eigenvalues are presented in Theorems~\ref{thDS5},~\ref{thDS4}. Theorem~\ref{thDS5} is formulated in terms of the (local) Fourier coefficients (\ref{2.18}) of the convolution kernel and similar coefficients (\ref{2.9}) of the potential.
The second result of Theorem~\ref{thDS5} is an upper bound for the lowest eigenvalue, see (\ref{2.17}).

Since the results of Theorem~\ref{thDS5} are expressed in terms of the local Fourier coefficients, it gives a nice opportunity to construct plenty of examples of  convolution kernels and potentials satisfying the assumptions of this theorem. Indeed, we can fix some $r$ and a sequence of the Fourier coefficients $a_n$ ensuring required conditions, and then  define the convolution kernel $a$ as a sum of Fourier series:
\begin{equation*}
a(x)=\sum\limits_{n\in\mathds{Z}^d} a_n e^{\frac{\pi\iu}{r}n\cdot x}\quad\text{on}\quad Q_{2r}(0)
\end{equation*}
and $a$ is arbitrary outside $Q_{2r}(0)$. The potential $V$ can be constructed in the same way via the Fourier coefficients defined in (\ref{2.9}).

For sufficiently smooth convolution kernels a lower bound for the number of discrete eigenvalues of $\mathcal{L}$ can be  also formulated
in terms of the derivatives of $a$ at zero. This is the subject of Theorem~\ref{thDS4}.
Here  again
it is possible to construct plenty of examples of $a$ and $V$ to which Theorem ~\ref{thDS4} applies: we can fix the derivatives $\p^n a(0)$ satisfying the assumptions of this Theorem and  define then the convolution kernel $a$ in the vicinity of zero as a polynomial with the prescribed derivatives. In view of definition (\ref{2.13}) of $h_N(\d)$, this function satisfies identity (\ref{2.14}) provided the potential approaches its global minimum quite fast.

Once we are given a generic smooth convolution kernel $a$, the corresponding form $\fa_N$ defined in (\ref{2.12})  might be quite bulky and it could be technically difficult to check whether this form is negative definite or not. In particular, the standard Sylvester criterion does not seem helpful at this point. This is why in
Theorem~\ref{thDS6} we provide some sufficient conditions guaranteeing the negative definiteness of the form $\fa_N$ and the validity of identity (\ref{2.14}). The latter identity is ensured by estimate (\ref{2.22}), while the negative definiteness of the form is due to estimates (\ref{2.19}), (\ref{2.20}) and one of inequalities (\ref{2.21_bis}), (\ref{2.21}). Both these inequalities  mean that the diagonal entries in the matrix of the form $\fa_N$ dominate the other entries.  Inequality (\ref{2.21_bis}) is more adapted to the case, when
$\#\mathds{I}$ is large enough, it states that in each line of the matrix of the form $\fa_N$ the contribution of the diagonal entry dominates
the contribution of all other elements in the same line.  Condition (\ref{2.21}) works better in the case when the cardinality of $\mathds{I}$
is small or the $n$th order derivatives of $a$ at zero  grow extremely fast (like $\exp({n^{2+\delta}})$) as $n\to\infty$.

Theorem~\ref{thDS6} is an efficient tool for checking the negativity of the form $\fa_N$ in various situations. For instance, if condition (\ref{2.19}) holds for at least one $n\in\mathds{Z}^d$,  we simply let $\mathds{I}:=\{n\}$. Then, if condition (\ref{2.14}) holds with $N=|n|$, we conclude immediately that the operator $\cL$ possesses at least one  eigenvalue below $\mu_0$.

Another way is to assume that conditions (\ref{2.19}), (\ref{2.20}) hold for all $n\in \mathds{Z}_+^d$ with $|n|\leqslant N$. Then we let $\mathds{I}=\{n\in\mathds{Z}_+^d:\, |n|\leqslant N\}$ and $S=\mathds{C}^{M(N)}$ and we see that the operator $\cL$ possesses at least $M(N)$ discrete eigenvalues below $\mu_0$.

Theorem~\ref{thDS6} can be also employed for identifying the situations with infinitely many discrete eigenvalues below $\mu_0$. Here we should assume that $a\in C^\infty(Q_r(0))$ and condition~(\ref{2.22}) holds with an arbitrary large $\a$, i.e.,
$|V(x)-V(x_0)|=o(|x-x_0|^\a)$ as $x$ approaches $x_0$.
Assume furthermore that there exists an infinite  subset $\mathds{I}\subseteq \mathds{Z}_+^d$ with
such that conditions (\ref{2.19}), (\ref{2.20}) hold for all $m,n\in\mathds{I}$ and at least one of inequalities (\ref{2.21_bis}), (\ref{2.21}) holds for each subset $\mathds{I}_N=\{n\in\mathds{I}\,:\,|n|\leqslant N\}$.
Then the assumptions of Theorem~\ref{thDS6} are  fulfilled for each subset  $\mathds{I}_N$ and, since $\#\mathds{I}_N $ grows unboundedly as $N\to\infty$, we conclude that the operator $\cL$ possesses infinitely many discrete eigenvalues below $\mu_0$ and then these eigenvalues necessarily accumulate to $\mu_0$.

Observe that functions $a(\cdot)$, for which (\ref{2.19}), (\ref{2.20}) hold for all $m,n\in\mathds{Z}_+^d$ and $\b_{n,m}=\b^{|n-m|}$ with $\b<1$ can not be analytic at zero. The reason is that condition~(\ref{2.20}) requires a very fast growth of the derivatives as $n$  increases.
Indeed, choosing one of the coordinate directions $x_j$, we derive from (\ref{2.20}) that for each  $k\in\mathds Z_+$ the inequality
$$
\Big|\frac{\partial^{2k+2}}{\partial x_j^{2k+2}}a(0)\Big| \Big(\Big|\frac{\partial^{2k}}{\partial x_j^{2k}}a(0)\Big| \Big)^{-1}
\geqslant \frac1{\b^4}
\Big|\frac{\partial^{2k}}{\partial x_j^{2k}}a(0)\Big| \Big(\Big|\frac{\partial^{2k-2}}{\partial x_j^{2k-2}}a(0)\Big| \Big)^{-1}
$$
holds true.
Iterating this inequality, we obtain
$$
\Big|\frac{\partial^{2k}}{\partial x_j^{2k}}a(0)\Big| \geqslant \Big(\frac1\b\Big)^{2k(k-1)}|a(0)|.
$$
Since  $a(0)\not=0$ due to (\ref{2.19}), the Taylor series of $a(\cdot)$ about zero does not converge for any $x\ne0$.

Ii is also possible to  construct a very rich class of examples of analytic at zero functions $a$, for which the operator $\cL$ possesses infinitely many discrete eigenvalues below $\mu_0$.  We provide a way of doing this
in Theorem~\ref{th_suf-infin}. Since the exponent $\g$ in (\ref{con2}), (\ref{con3}) can be less than one, we see easily that there is a wide class of analytic at zero convolution kernels obeying the assumptions of Theorem~\ref{th_suf-infin}.

As an example, we consider the one-dimensional case and let
\begin{equation*}
a(z)=-(1+z^2)^{-1}, \qquad V(x)=e^{-x^{-2}}-5.
\end{equation*}
Since in the vicinity of zero the function $a(z)$ admits a representation
\begin{equation*}
a(z)=-\sum\limits_{j=0}^{\infty}(-z^2)^j,
\end{equation*}
then for all $n\in\mathds{Z}^+$ we have
$$
(-1)^{n}\p^{2n}a(0)=-(2n)!,\qquad \p^{2n+1}a(0)=0.
$$
Hence, conditions (\ref{con1}), (\ref{con2}), (\ref{con3}) of Theorem~\ref{th_suf-infin} are satisfied and it follows from   the definition of
$V(\cdot)$ that   other conditions of this theorem are also fulfilled. We then conclude that the operator $\mathcal{L}$ with the convolution kernel $a(\cdot)$ and the potential $V(\cdot)$ has infinitely many eigenvalues in the interval $[-5-\pi,-5)$.

In a higher dimension $d\geqslant 1$ we can choose
$$
a(z)=-\prod\limits_{k=1}^d \frac1{1+z_k^2}\qquad\text{or}\qquad a(z)=\frac 1{1+|z|^{2d}}
$$
and these kernels also satisfy the assumptions of Theorem~\ref{th_suf-infin}.

Theorem~\ref{thDS2} provides some more sufficient conditions for the existence of
infinitely many eigenvalues. These conditions are formulated either in terms of the range of the Fourier transform of the convolution kernel, see Item~\ref{thDS2it1} or via the local Fourier coefficients, see Item~\ref{thDS2it2}. Observe that these local Fourier coefficients are exactly the ones previously used in Theorem~\ref{thDS5}. It should be also said that conditions (\ref{2.10}) are equivalent to the condition  that the function $\hat{a}$ is non-positive and is not identically zero. We emphasize that in the formulation  of Theorem~\ref{thDS2} it is  assumed that the potential $V$  equals identically to its global minimum in some neighbourhood of the point $x_0$. This condition is crucial.

Our final Theorem~\ref{thDS3} provides an upper bound for the number of the discrete eigenvalues. Its proof is based on an appropriate adaption of the classical Birman-Schwinger principle. The final upper bound is rather different  in comparison with the classical result, namely, here both the convolution kernel and the potential contribute to the bound via the integrals $I_a$ and $I_V$. We also see that the integral $I_V$ is finite only provided the potential $V$ does not approach its global minimum very fast and this is in a good agreement with the above discussed theorems treating the cases of infinitely many eigenvalues.

In view of the above discussed statements  we observe an important fact: the number of discrete eigenvalues of $\mathcal{L}$ depends essentially on how the potential $V$ approaches its global minimum. The faster it tends to this minimum, the more discrete eigenvalues are present. In particular, according to Theorems~\ref{thDS5},~\ref{th_suf-infin}, the operator $\cL$ can have infinitely many eigenvalues provided the potential $V$ approaches its global minimum either exponentially fast (in Theorem~\ref{thDS5}) or it
coincides with this minimum identically in some   neighbourhood of  $x_0$ (in Theorem~\ref{th_suf-infin}). And vice versa, if the potential $V$ approaches its global minimum very slowly then conditions (\ref{2.8}), (\ref{2.16}), (\ref{2.14}), (\ref{2.22}) are violated and we can not guarantee even the existence of the discrete spectrum. Moreover, in this case Theorem~\ref{thDS3} says that the operator $\cL$ can have only finitely many eigenvalues.

This explains why   Schr\"odinger operators with localized potentials typically have finitely many eigenvalues below the bottom of the essential spectrum, see \cite{Simon1},  \cite{Klaus}, \cite{Klaus2}, \cite{ExnKov}, and to get infinitely many eigenvalues, one has to assume that the localized potential should decay at infinity quite slowly, see \cite{Simon2}. Indeed, given a one-dimensional Schr\"odinger  operator with a localized potential
\begin{equation*}
\cH:=-\frac{d^2\ }{dx^2}+V(x),
\end{equation*}
we make its Fourier transform getting then the operator
\begin{equation*}
\hat{\cH}=\cL_{\xi^2}+\cL_{\hat{V}\star},\qquad \hat{V}:=\cF^{-1}[V].
\end{equation*}
Here the second derivative becomes the operator  of multiplication by $\xi^2$ and only this part of the operator $\hat{\cH}$ fully determines the essential spectrum, which is $[0,+\infty)$. The function $\xi\mapsto \xi^2$ approaches its global minimum, which is zero, with a fixed rate, and exactly this prevents the existence of infinitely many eigenvalues for typical localized potentials $V$. In view of this fact, we can state that in the case of non-local Schr\"odinger operators we impose no apriori restrictions for the behavior of the potential $V$ in the vicinity of its global minimum and this is why the variety of possible spectral pictures is much richer than in the case of differential Schr\"odinger operators.

In conclusion of this section, we shortly discuss some applications of our results to the population dynamics models mentioned in the Introduction.  The large time behaviour of the population depends crucially on whether the operator $\mathcal{L}-\langle a\rangle$ on the right-hand side of (\ref{intro-4}) has a positive eigenvalue or not. In the former case the population exhibits an exponential growth, and its asymptotic profile  is proportional to the principal eigenfunction. Moreover, the rate of stabilization to this profile is determined by the distance form the principal positive
eigenvalue to the rest of the spectrum.

Under the assumption that $\langle a\rangle=1$ the operator $\mathcal{L}-I$ has a positive eigenvalue  if and only if the operator $\mathcal{L}$
has a point of the discrete spectrum above $1$.   Due to the biological interpretation of the potential $V$,  the inequality $V\leqslant 1$
should be satisfied.  If $\max V=1$ then the existence of a positive eigenvalue of $\mathcal{L}-I$  is governed by condition
(\ref{2.8a}). Otherwise, we should consider the Fourier image of $\mathcal{L}$ and apply our results to the transformed operator.

\section{Essential spectrum}

In this section we prove Theorem~\ref{thES}. We begin with an auxiliary lemma.

\begin{lemma}\label{lmL}
The operator $\cL$ is bounded and self-adjoint in $L_2(\mathds{R}^d)$.
\end{lemma}

\begin{proof}
We introduce two auxiliary operators in $L_2(\mathds{R}^d)$ by the formulae
\begin{equation*}
(\cL_{a\star}u)(x):=\int\limits_{\mathds{R}^d} a(x-y)u(y)\,dy,\qquad (\cL_{V} u)(x):=V(x)u(x).
\end{equation*}
Since the function $V$ is bounded and real-valued, we immediately conclude that the operator $\cL_V$ is bounded and symmetric in $L_2(\mathds{R}^d)$, and hence, it is self-adjoint.

Employing the fact that $u\in L_1(\mathds{R}^d)$, for each $u\in L_2(\mathds{R}^d)$ by the Cauchy-Schwarz inequality we have:
\begin{equation}\label{3.0}
\begin{aligned}
\|\cL_{a\star} u\|_{L_2(\mathds{R}^d)}^2\leqslant & \int\limits_{\mathds{R}^d} \,dx \left(\int\limits_{\mathds{R}^d} |a(x-y)||u(y)|\,dy\right)^2
\\
\leqslant &\int\limits_{\mathds{R}^d} \int\limits_{\mathds{R}^d} |a(x-y)|\,dy \int\limits_{\mathds{R}^d} |a(x-y)||u(y)|^2\,dy
\\
\leqslant & \|a\|_{L_1(\mathds{R})} \int\limits_{\mathds{R}^{2d}} |a(x-y)||u(y)|^2\,dxdy
\\
=&\|a\|_{L_1(\mathds{R})} \int\limits_{\mathds{R}^{2d}} |a(x)||u(y)|^2\,dxdy
=\|a\|_{L_1(\mathds{R})}^2 \|u\|_{L_2(\mathds{R}^d)}^2.
\end{aligned}
\end{equation}
This proves the boundedness of the operator $\cL_{a\star}$. The symmetricity, and hence, the self-adjointness, is confirmed straightforwardly by means of assumption (\ref{2.1}):
\begin{align*}
(\cL_{a\star}u,v)_{L_2(\mathds{R}^d)} = \int\limits_{\mathds{R}^{2d}} a(x-y) u(y) \overline{v(x)}\,dxdy=\int\limits_{\mathds{R}^{2d}} u(y) \overline{v(x)a(y-x)}\,dxdy=(u,\cL_{a\star}v)_{L_2(\mathds{R}^d)}.
\end{align*}
The proof is complete.
\end{proof}

The rest of this section is devoted to the proof of Theorem~\ref{thES}.
It is straightforward to confirm that under the unitary Fourier transform the bounded self-adjoint operators $\cL_{a\star}$ and $\cL_V$ are unitarily equivalent respectively to the operator of multiplication by $\hat{a}$ and to the operator of convolution with $\hat{V}$. Namely, the identities hold:
\begin{equation}\label{3.2}
\left(\frac{1}{(2\pi)^\frac{d}{2}}
\cF\right) \cL_{a\star} \left(\frac{1}{(2\pi)^\frac{d}{2}}
\cF\right)^{-1}=\cL_{\hat{a}}, \qquad
\left(\frac{1}{(2\pi)^\frac{d}{2}}
\cF\right) \cL_{V} \left(\frac{1}{(2\pi)^\frac{d}{2}}
\cF\right)^{-1}=\cL_{\hat{V}\star}.
\end{equation}

In view of the continuity of the functions $V$ and $\hat{a}$, the spectra of the operators $\cL_{\hat{a}}$ and $\cL_{V}$ coincides with their essential parts and are given by the following identities:
\begin{equation}\label{3.3}
\spec(\cL_{\hat{a}})=\essspec(\cL_{\hat{a}})=[a_{\min},a_{\max}], \qquad
\spec(\cL_V)=\essspec(\cL_{V})=[V_{\min},V_{\max}].
\end{equation}
Hence, by identities (\ref{3.2}), the same is true for the operators $\cL_{a\star}$ and $\cL_{\hat{V}\star}$:
\begin{equation}\label{3.4}
\spec(\cL_{a\star})=\essspec(\cL_{a\star})=[a_{\min},a_{\max}], \qquad
\spec(\cL_{\hat{V}\star})=\essspec(\cL_{\hat{V}\star})=[V_{\min},V_{\max}].
\end{equation}
We also observe an obvious identity
\begin{equation*}
\cL=\cL_{a\star} + \cL_{V}.
\end{equation*}

Our next step is to prove the inclusion
\begin{equation}\label{3.6}
\essspec(\cL_{a\star})\cup\essspec(\cL_{V})\subseteq \essspec(\cL).
\end{equation}
We introduce a family of functions:
\begin{equation}\label{3.7}
\phi_\d(x):=\left\{
\begin{aligned}
&\d^{-\frac{d}{2}} && \quad\text{on}\quad Q_\d(0),
\\
& 0 && \quad\text{outside}\quad Q_\d(0),
\end{aligned}\right.
\end{equation}
where $\d$ is supposed to be small enough. Then we choose arbitrary $\l\in
(V_{\min},V_{\max})$ and by the continuity of $V$ we conclude that there exists $x_0\in\mathds{R}^d$ such that $V(x_0)=\l$. By straightforward calculations we then easily confirm that
\begin{equation}\label{3.8}
\big\|(\cL_V-\l)\phi_\d(\,\cdot\,-x_0)\big\|_{L_2(\mathds{R}^d)}^2=\d^{-d} \int\limits_{Q_\d(x_0)} |V(x)-V(x_0)|^2\,dx\to 0,\qquad \d\to+0.
\end{equation}
We also observe that the family $\{\phi_\d(x-x_0)\}$ is non-compact and $\|\phi_d(\,\cdot\,-x_0)\|_{L_2(\mathds{R}^d)}=1$ for all $\d$ and $x_0$. Hence, each sequence $\phi_{\d_n}(\,\cdot\,-x_0)$ with arbitrary sequence $\d_n\to+0$, $n\to\infty$, is a Weyl sequence for the operator $\cL_V$ at the point $\l$. If we prove that
\begin{equation}\label{3.9}
\cL_{a\star}\phi_\d(\,\cdot\,-x_0)\to 0,\qquad \d\to+0,
\end{equation}
then together with (\ref{3.8}) this will imply that the sequence $\phi_{\d_n}(\,\cdot\,-x_0)$ is also a Weyl one for the operator $\cL$ at the point $\l$ and hence,
\begin{equation}\label{3.10}
\essspec(\cL_{V})\subseteq \essspec(\cL).
\end{equation}

We prove (\ref{3.9}) by rather straightforward calculations. Namely,
\begin{equation}\label{3.12}
\|\cL_{a\star}\phi_\d(\,\cdot\,-x_0)\|_{L_2(\mathds{R}^d)}^2\leqslant \d^{-d}
\int\limits_{\mathds{R}^d}\,dx \left(\int\limits_{Q_\d(x_0)} |a(x-y)|\,dy\right)^2\leqslant \d^{-d} J_\d \int\limits_{\mathds{R}^d}\,dx \int\limits_{Q_\d(x-x_0)} |a(y)|\,dy,
\end{equation}
where we have denoted
\begin{equation*}
J_d:=\sup\limits_{x\in\mathds{R}^d}\int\limits_{Q_\d(x_0)} |a(x-y)|\,dy =\sup\limits_{x\in\mathds{R}^d}\int\limits_{Q_\d(x-x_0)} |a(y)|\,dy \leqslant \|a\|_{L_1(\mathds{R}^d)}.
\end{equation*}
Since the measures of the set $Q_\d(x-x_0)$ are equal to $\d^d$ for all $x-x_0$ and the function $|a|$ is integrable over $\mathds{R}^d$, by the absolute continuity of the Lebesgue integral we conclude that
\begin{equation}\label{3.11}
J_\d\to 0,\qquad \d\to+0.
\end{equation}
Then we can continue estimating in (\ref{3.12}) as follows:
\begin{equation*}
\|\cL_{a\star}\phi_\d(\,\cdot\,-x_0)\|_{L_2(\mathds{R}^d)}^2 \leqslant \d^{-d} J_\d \int\limits_{\mathds{R}^d}\,dy |a(y)|\int\limits_{Q_\d(x_0+y)} \,dx= J_\d \|a\|_{L_1(\mathds{R}^d)}
\end{equation*}
and by (\ref{3.11}) we then arrive at (\ref{3.9}) and hence, to (\ref{3.10}).
In view of unitary equivalence (\ref{3.2}) and identities (\ref{3.3}), (\ref{3.4}) we then get that $\essspec(\cL_{a\star})\subseteq \essspec(\cL)$ and together with (\ref{3.10}) this leads us to (\ref{3.6}).

To complete the proof of identity
\begin{equation}\label{3.19}
\essspec(\cL)=\essspec(\cL_{a\star})\cup\essspec(\cL_{V})=[\mu_0,\mu_1],
\end{equation}
it sufficient to show that
\begin{equation*}
\essspec(\cL)\setminus \big(
\essspec(\cL_{a\star})\cup\essspec(\cL_{V})\big)=\emptyset.
\end{equation*}

Let $\l\in\essspec(\cL)$ and $\l\not\in
\essspec(\cL_{a\star})$, $\l\not\in\essspec(\cL_{V})\big)$. Then there exists a Weyl sequence $u_n\in L_2(\mathds{R}^d)$, which is bounded, non-compact and
\begin{equation}\label{3.14}
f_n:=(\cL-\l)u_n\to 0,\qquad n\to\infty.
\end{equation}

Since $\l\notin\essspec(\cL_V)$, by the second identity in (\ref{3.3}), the inverse operator $(\cL_V-\l)^{-1}$ is well-defined and bounded. We hence can rewrite (\ref{3.14}) as
\begin{equation}\label{3.15}
\frac{1}{V-\l}\cL_{a\star}u_n+u_n=\frac{f_n}{V-\l}\to0,\qquad n\to+\infty.
\end{equation}

In view of (\ref{2.7}), zero belongs to the essential spectrum of the operator $\cL_V$ and hence, $\l\ne 0$, $V-\l\ne0$. Then
\begin{equation*}
\frac{1}{V-\l}=-\frac{1}{\l} + \frac{V_1}{\l},\qquad V_1:=\frac{V}{V-\l}.
\end{equation*}
We substitute this identity into (\ref{3.15}) and we get:
\begin{equation}\label{3.17}
(\cL_{a\star}-\l) u_n + V_1 \cL_{a\star} u_n=\frac{\l}{V-\l}f_n.
\end{equation}
By our assumptions and by (\ref{3.4}) the number $\l$ is in the resolvent set of the operator $\cL_{a\star}$ and hence, the resolvent $(\cL_{a\star}-\l)^{-1}$ is well-defined and is a bounded in $L_2(\mathds{R}^d)$. This allows us to rewrite (\ref{3.17}) as
\begin{equation}\label{3.18}
u_n=(\cL_{a\star}-\l)^{-1} \left(\frac{\l}{V-\l}f_n- V_1 \cL_{a\star} u_n\right).
\end{equation}

According to our assumptions on $V$, this function decays at infinity. Hence, the same is true for $V_1$. Then it is easy to see that the operator $V_1 \cL_{a\star}$ is compact in $L_2(\mathds{R}^d)$. Since the sequence $u_n$ is bounded, it contains a subsequence, still denoted by $u_n$, such that $V_1\cL_{a\star}u_n$ converges in $L_1(\mathds{R}^d)$. The sequence $\frac{\l}{V-\l}f_n$ also converges as $n\to+\infty$; the limiting function is zero. Hence, the right hand side in (\ref{3.18}) is a converging sequence as $n\to+\infty$. This contradicts the non-compactness of the sequence $u_n$.
Hence, identity (\ref{3.19}) holds and this proves the first part of the theorem.

We proceed to proving the second part of the theorem. In view of the first identity in (\ref{3.2}), the quadratic form associated with the operator $\cL$ reads
\begin{equation}\label{3.20}
\fl[u]:=(\cL u,u)_{L_2(\mathds{R}^d)}=(\cL_{a\star}u,u)_{L_2(\mathds{R}^d)}+ (V u,u)_{L_2(\mathds{R}^d)}=(\cL_{\hat{a}}\hat{u},\hat{u})_{L_2(\mathds{R}^d)}+ (V u,u)_{L_2(\mathds{R}^d)},
\end{equation}
where
\begin{equation*}
\hat{u}:=\frac{1}{(2\pi)^{\frac{d}{2}}} \cF[u],\qquad \|u\|_{L_2(\mathds{R}^d)}=\|\hat{u}\|_{L_2(\mathds{R}^d)}.
\end{equation*}
Hence, identity (\ref{3.20}) implies immediately that this form satisfies the estimate
\begin{equation*}
(a_{\min}+V_{\min})\|u\|_{L_2(\mathds{R}^d)}^2\leqslant (\cL u,u)_{L_2(\mathds{R}^d)}\leqslant
(a_{\max}+V_{\max})\|u\|_{L_2(\mathds{R}^d)}^2
\end{equation*}
and hence, the spectrum of the operator $\cL$ is located inside the segment $[a_{\min}+V_{\min},a_{\max}+V_{\max}]$. Now the second part of the theorem follows from the standard properties of the spectra of self-adjoint operators and identity (\ref{3.19}). This completes the proof of Theorem~\ref{thES}.

\section{Existence of discrete spectrum}

In this section we study the existence of the discrete spectrum of the operator $\cL$, namely, we prove Theorem~\ref{thDS1}.

\subsection{Proof of Theorem~\ref{thDS1}}

The proof is based on the minimax principle: if we find a normalized function $\vp\in L_2(\mathds{R}^d)$ such that
$\fl[\vp]<V_{min}$,
this will imply the statement of the theorem; we recall that $\fl[u]$ is the quadratic form associated with the operator $\cL$, see (\ref{3.20}).

We construct a required test function explicitly choosing it to be $\phi_\d(x-x_0)$ with $\phi_\d$ introduced in (\ref{3.7}) and $\d$ mentioned in the formulation of the theorem; we note that this function is normalized in $L_2(\mathds{R}^d)$. Having this normalization in mind, we consider the quadratic form $\fl$ on such function, namely:
\begin{equation}\label{4.1}
\begin{aligned}
\fl[\phi_\d(\,\cdot\,-x_0)]-V_{\min}
= & \d^{-d} \int\limits_{Q_\d(x_0\times Q_\d(x_0)} a(x-y) \,dxdy
 + \d^{-d} \int\limits_{Q_\d(x_0)} \big(V(x)-V_{\min}\big)\,dx
\\
=& \d^{d} \int\limits_{Q_1(0)\times Q_1(0)} a(\d(x-y)) \,dxdy
 + \int\limits_{Q_1(0)} \big(V(x_0+\d x)-V_{\min}\big)\,dx.
\end{aligned}
\end{equation}
Let us calculate the first integral in the above identity.

First of all observe that owing to condition (\ref{2.1}) we immediately get
\begin{align*}
\int\limits_{Q_1(0)\times Q_1(0)} a(\d(x-y)) \,dxdy =&
\int\limits_{Q_1(0)\times Q_1(0)} a(\d(y-x)) \,dxdy
\\
=& \int\limits_{Q_1(0)\times Q_1(0)} \overline{a(\d(x-y))} \,dxdy =\int\limits_{Q_1(0)\times Q_1(0)} \RE a(\d(x-y)) \,dxdy.
\end{align*}
Then we make the change of the variables $(x,y)\to(x-y,x+y)$:
\begin{align*}
\int\limits_{Q_1(0)\times Q_1(0)} \RE a(\d(x-y)) \,dxdy=&2^{-d} \int\limits_{Q_2(0)}\,dx \RE a(\d x) \int\limits_{\{y:\, |y_i|<1-|x_i|,\, i=1,\ldots,d\}} \,dy
\\
=& \int\limits_{Q_{2}(0)} \prod\limits_{i=1}^{d} (1-|x_i|) \RE a(\d x)\,dx.
\end{align*}
Now by (\ref{4.1}) we get:
\begin{align*}
\fl[\phi_\d(\,\cdot\,-x_0)]-V_{\min}\|\phi_\d(\,\cdot\,-x_0)\|_{L_2(\mathds{R}^d)}^2 = \d^d
\bigg(& \int\limits_{Q_2(0)} \prod\limits_{i=1}^{d} (1-|x_i|) \RE a(\d x)\,dx
\\
&+
 \d^{-d} \int\limits_{Q_1(0)} \big(V(x_0+\d x)-V_{\min}\big)\,dx
\bigg)<0.
\end{align*}
Hence, by the minimax principle we conclude that the operator $\cL$ has a non-empty discrete spectrum below $V_{\min}$.
This completes the proof of Theorem~\ref{thDS1}.

\section{Existence of finitely many eigenvalues}

In this section we discuss sufficient conditions ensuring the existence of at least finitely many eigenvalues of the operator $\cL$, namely, we prove Theorems~\ref{thDS5} and~\ref{thDS4}.

\subsection{Proof of Theorem~\ref{thDS5}}

Since the restriction of the function $a$ on $Q_{2r}(0)$ belongs to $L_1(Q_{2r}(0))$, for each $\eta>0$ there exists an infinitely differentiable function $a^\eta\in C_0^\infty(Q_{2r}(0))$ such that
\begin{equation}\label{4.3}
\|a-a^\eta\|_{L_1(Q_{2r}(0))}\leqslant \eta.
\end{equation}

Let $u=u(x)$ be an infinitely differentiable function on $\overline{Q_r(x_0)}$; we extend it by zero outside $\overline{Q_r(x_0)}$. Then by estimates (\ref{3.0}) and (\ref{4.3}) we find:
\begin{equation}\label{4.4}
\begin{aligned}
 (\cL_{a\star}u,u)_{L_2(\mathds{R}^d)}
=& \int\limits_{Q_r(x_0)\times Q_r(x_0)} a(x-y)u(y)\overline{u(x)}\,dxdy
\\
=&
\int\limits_{Q_r(0)\times Q_r(0)} a^\eta(x-y) u(y+x_0) \overline{u(x+x_0)}\,dxdy
+\fr^\eta[u],
\end{aligned}
\end{equation}
where $\fr^\eta[u]$ is a quadratic form satisfying the estimate
\begin{equation}\label{4.5a}
|\fr^\eta[u]|\leqslant \eta \|u\|_{L_2(\mathds{R})^d}^2.
\end{equation}
We represent the function $a^\eta$ by its Fourier series, namely,
\begin{equation*}
a^\eta(x)=\sum\limits_{n\in\mathds{Z}^d} a_n^\eta e^{\frac{\iu\pi}{r}n\cdot x},\quad x\in Q_{2r}(0),\qquad a_n^\eta:=(2r)^{-d} \int\limits_{Q_{2r}(0)} a^\eta(x)e^{-\frac{\iu\pi}{r}n\cdot x}\,dx.
\end{equation*}
It follows from (\ref{4.3}) that
\begin{equation*}
|a_n^\eta-a_n|\leqslant \eta,\qquad n\in\mathds{Z}^d, \qquad a_n:=(2r)^{-d}\int\limits_{Q_{2r}(0)} a(x) e^{-\frac{\iu \pi}{r} n\cdot x}\,dx.
\end{equation*}

Owing to the assumed smoothness of the function $a^\eta$, its Fourier series converges uniformly on $\overline{Q_{2r}(0)}$. This allows us to substitute this Fourier series into the first term on the right hand side of (\ref{4.4}):
\begin{equation}\label{4.6}
\begin{aligned}
\int\limits_{Q_r(0)\times Q_r(0)} &a^\eta(x-y) u(x_0+y)\overline{u(x_0+x)}\,dxdy
\\
=& \sum\limits_{n\in\mathds{Z}^d} a_n^\eta \int\limits_{Q_r(0)\times Q_r(0)}
e^{\frac{\iu\pi}{r}n\cdot (x-y)}u(x_0+y) \overline{u(x_0+x)}\,dxdy
 =\sum\limits_{n\in\mathds{Z}^d} a_n^\eta |U_n|^2,
\end{aligned}
\end{equation}
where
\begin{equation*}
U_n:=\int\limits_{Q_r(0)} e^{-\frac{\iu\pi}{r}n\cdot x} u(x_0+x)\,dx= e^{\frac{\iu\pi}{r}n\cdot x_0} \int\limits_{Q_{2r}(x_0)} e^{-\frac{\iu\pi}{r}n\cdot x} u(x)\,dx.
\end{equation*}
Here we have also employed that $u$ vanishes on $Q_{2r}(x_0)\setminus Q_r(x_0)$.
Up to a fixed multiplicative constant, the numbers $U_n$ are the Fourier coefficients of the function $u$, namely,
\begin{equation*}
u(x+x_0)=(2r)^{-d}\sum\limits_{n\in\mathds{Z}^d} U_n e^{\frac{\iu\pi}{r}n\cdot x},\qquad x\in Q_{2r}(0),
\end{equation*}
and by the Parseval identity holds:
\begin{equation}\label{4.9b}
\|u\|_{L_2(Q_{2r}(x_0))}^2=\|u\|_{L_2(Q_r(x_0))}^2= \|u(\,\cdot\,+x_0)\|_{L_2(Q_r(0))}^2=(2r)^{-d} \sum\limits_{n\in\mathds{Z}^d}
|U_n|^2.
\end{equation}

The function $u(x)$ can be also regarded as defined on the cube $Q_r(x_0)$ and it can be represented by one more Fourier series
\begin{equation*}
u(x+x_0)=r^{-d}\sum\limits_{n\in\mathds{Z}^d} u_n e^{\frac{2\pi\iu}{r}n\cdot x},\qquad x\in Q_r(0),\qquad u_n:=\int\limits_{Q_r(0)} e^{-\frac{2\pi\iu}{r}n\cdot x}u(x+x_0)\,dx.
\end{equation*}
The corresponding Parseval identity holds true:
\begin{equation}\label{4.32}
\|u\|_{L_2(Q_r(x_0))}^2=\|u(\,\cdot\,+x_0)\|_{L_2(Q_r(0))}^2= r^{-d}\sum\limits_{n\in\mathds{Z}^d} |u_n|^2.
\end{equation}
We also observe the identity
\begin{equation}\label{4.33}
U_{2n}=u_n,\qquad n\in\mathds{Z}^d,
\end{equation}
which will play an important role in what follows.

We substitute (\ref{4.6}) into (\ref{4.4}) and take into consideration estimate (\ref{4.5a}) and Parseval identity (\ref{4.9b}). This gives:
\begin{align*}
(\cL_{a\star}u,u)_{L_2(\mathds{R}^d)} \leqslant & (\cL_{a^\eta\star}u,u)_{L_2(\mathds{R}^d)}+\eta\|u\|_{L_2(Q_r(x_0))}^2 = \sum\limits_{n\in\mathds{Z}^d} a_n^\eta |U_n|^2 + \eta\|u\|_{L_2(Q_r(x_0))}^2
\\
\leqslant & \sum\limits_{n\in\mathds{Z}^d} a_n |U_n|^2 + \eta \big(1+(2r)^d\big)\|u\|_{L_2(Q_r(x_0))}^2.
\end{align*}
Passing then to the limit as $\eta\to+0$ and using identity (\ref{4.33}), we get
\begin{equation}\label{4.35}
(\cL_{a\star}u,u)_{L_2(\mathds{R}^d)} \leqslant \sum\limits_{n\in\mathds{Z}^d} a_n|U_n|^2=\sum\limits_{n\in(2\mathds{Z})^d} a_{2n}|u_n|^2 + \sum\limits_{n\in\mathds{Z}^d\setminus (2\mathds{Z})^d} a_n|U_n|^2.
\end{equation}
By the
Parseval identity (\ref{4.9b}) we obtain:
\begin{align*}
\sum\limits_{n\in \mathds{Z}^d \setminus (2\mathds{Z})^d} a_n |U_n|^2 \leqslant \sup\limits_{n\in\mathds{Z}^d\setminus (2\mathds{Z})^d} a_n \sum\limits_{n\in\mathds{Z}^d\setminus (2\mathds{Z})^d}|U_n|^2 \leqslant \a \|u\|_{L_2(Q_r(x_0))}^2,\qquad \a:=(2r)^d \sup\limits_{n\in\mathds{Z}^d\setminus (2\mathds{Z})^d} a_n;
\end{align*}
here we have also used the inequality $\alpha\geqslant 0$.
This allows us to rewrite estimate (\ref{4.35}) as
\begin{equation}\label{4.36}
(\cL_{a\star}u,u)_{L_2(\mathds{R}^d)} \leqslant \sum\limits_{n\in\mathds{Z}^d} a_{2n}|u_n|^2 + \a \|u\|_{L_2(Q_r(x_0))}^2.
\end{equation}

Now let us consider test functions $u\in L^2(\mathds R^d)$ supported in the cube $Q_r(x_0)$. We suppose that on
$Q_r(x_0)$ the function $u$ is a finite linear combination
\begin{equation}\label{4.37}
u(x)=r^{-d}\sum\limits_{n\in \mathds{J}} u_n e^{\frac{2\pi\iu}{r}n\cdot (x-x_0)},
\end{equation}
where $\mathds{J}$ is a finite subset of $\mathds{J}_0$. Then for such $u$ estimate (\ref{4.36}) becomes
\begin{equation}\label{4.38}
(\cL_{a\star}u,u)_{L_2(\mathds{R}^d)} \leqslant \sum\limits_{n\in \mathds{J}} a_{2n}|u_n|^2 +\a \|u\|_{L_2(Q_r(x_0))}^2.
\end{equation}
We also have:
\begin{align*}
\big((V(x)-V_{\min})u,u\big)_{L_2(\mathds{R}^d)}=&\int\limits_{Q_r(x_0)} (V(x)-V_{\min}) |u(x)|^2\,dx
\\
=& r^{-2d} \sum\limits_{m,n\in \mathds{J}} u_n \overline{u_m} \int\limits_{Q_r(0)} \big(V(x+x_0)-V_{\min}\big) e^{\frac{2\pi\iu}{r}(n-m)\cdot x}\,dx
 \\
 =& r^{-2d} \sum\limits_{n,m\in \mathds{J}} u_n \overline{u_m} V_{n-m}.
\end{align*}
Hence, by Cauchy-Schwartz inequality and Parseval identity,
\begin{align*}
\big((V(x)-V_{\min})u,u\big)_{L_2(\mathds{R}^d)} \leqslant &r^{-2d} \left(\sum\limits_{n\in \mathds{J}} |u_n|^2\right)^{\frac{1}{2}} \left(\sum\limits_{n\in \mathds{J}} \bigg|\sum\limits_{m\in \mathds{J}} V_{n-m}\overline{u_m}\bigg|^2\right)^{\frac{1}{2}}
\\
\leqslant & r^{-2d} \left(\sum\limits_{n\in \mathds{J}} |u_n|^2\right)^{\frac{1}{2}}
\left(\sum\limits_{n\in \mathds{J}} \left(\sum\limits_{m\in \mathds{J}} |V_{n-m}|\right)
\left(\sum\limits_{m\in \mathds{J}} |V_{n-m}||u_m|^2\right)\right)^{\frac{1}{2}}
\\
\leqslant & r^{-\frac{3}{2}d} \left(\sum\limits_{n\in \mathds{J}} |u_n|^2\right)^{\frac{1}{2}} \nu_\mathds{J}^{\frac{1}{2}} \left(\sum\limits_{m,n\in \mathds{J}} |V_{n-m}||u_n|^2\right)^{\frac{1}{2}} \leqslant \nu_\mathds{J} \|u\|_{L_2(Q_r(x_0))}^2.
\end{align*}
This estimate, (\ref{4.38}) and Parseval identity (\ref{4.32}) lead us to a final estimate for the form of the operator $\cL$ on the functions $u$ defined in (\ref{4.37}):
\begin{equation}\label{4.40}
\fl[u]-V_{\min}\|u\|_{L_2(\mathds{R}^d)}^2 \leqslant (r^d \max\limits_{n\in \mathds{J}} a_{2n} + \a + \nu_\mathds{J}) \|u\|_{L_2(Q_r(x_0))}^2.
\end{equation}

We substitute for $\mathds{J}$ in (\ref{4.40}) the set $\mathds{I}$ from the formulation of the theorem. Then combining (\ref{2.16}) and (\ref{4.40}) yields
\begin{equation*}
\fl[u]-V_{\min}\|u\|_{L_2(\mathds{R}^d)}^2 <0
\end{equation*}
for all linear combinations (\ref{4.37}) with $\mathds{J}\subset \mathds{I}$. By the minimax principle this implies that the operator $\cL$ possesses at least $\# \mathds{I}$ eigenvalues below $V_{\min}$ and this completes the proof of the first statement in the theorem.

Now as the set $\mathds{J}$ in (\ref{4.40}) we choose $\mathds{J}:=\{n\}$ with $n\in \mathds{I}$. Then estimate (\ref{4.40}) becomes
\begin{equation*}
\fl[u]-V_{\min}\|u\|_{L_2(\mathds{R}^d)}^2 \leqslant r^d a_{2n} + \a + r^{-d}|V_0|,\qquad |V_0|=V_0.
\end{equation*}
Taking the infimum over $n\in \mathds{J}$ of the right hand side in the above inequality, by the minimax principle we arrive at (\ref{2.17}). This completes the proof of Theorem~\ref{thDS5}.

\subsection{Proof of Theorem~\ref{thDS4}}

Let $U$ be an arbitrary function defined on the cube $Q_1(0)$ and being an element of $L_2(Q_1(0))$. We extend all such functions by zero outside $Q_1(0)$. Then we choose a sufficiently small $\d$ and let $u_\d(x):=U((x-x_0)\d^{-1})$. This function is supported in $Q_\d(x_0)$.
The quadratic form of the operator $\cL_{a\star}$ on the function $u_\d$ reads as
\begin{equation}\label{4.10}
\begin{aligned}
(\cL_{a\star}u_\d,u_\d)_{L_2(\mathds{R}^d)} = &\int\limits_{Q_\d(x_0)\times Q_\d(x_0)} a(x-y) u_\d(y)\overline{u_\d(x)}\,dxdy
\\
=&\d^{2d} \int\limits_{Q_1(0)\times Q_1(0)} a(\d(x-y)) U(y)\overline{U(x)}\,dxdy.
\end{aligned}
\end{equation}

Since the function $a$ is smooth, we can represent it by the Taylor formula as
\begin{equation}\label{4.11}
a(\d(x-y))=\sum\limits_{j=0}^{2N}\d^j A_j(\xi) + \d^{2N+1} \tilde{A}_{2N+1}(x-y,\d),
\end{equation}
where $A_j$ are homogeneous polynomials of degree $j$ given by the formulae
\begin{equation*}
A_j(\xi)=\sum\limits_{\substack{n\in\mathds{Z}_+^d \\ |n|=j}} \frac{\p^n a(0)}{n!}\xi^n.
\end{equation*}
The remainder $\tilde{A}_{2N+1}$ in (\ref{4.11}) satisfies the uniform estimate
\begin{equation}\label{4.13}
|\tilde{A}_{2N+1}(\xi,\d)|\leqslant C\quad\text{for all}\quad \xi\in \overline{Q_2(0)},
\end{equation}
where $C$ is some constant independent of $\xi$ and $\d$. Since
\begin{equation*}
(x-y)^n=\sum\limits_{\substack{m,q\in\mathds{Z}_+^d \\ m+q=n}} (-1)^{|q|} \frac{n!}{m!q!} x^m y^q,
\end{equation*}
we immediately get
\begin{equation*}
A_j(x-y) = \sum\limits_{\substack{m,q\in\mathds{Z}_+^d \\ |m|+|q|=j}} (-1)^{|q|} \p^{m+q} a(0)
\frac{x^m y^q}{m!q!}.
\end{equation*}
We substitute this formula into (\ref{4.11}) and the result is plugged in (\ref{4.10}). Denoting then
\begin{equation*}
U_m:=\int\limits_{Q_1(0)} x^m U(x)\,dx,
\end{equation*}
we arrive at the identities
\begin{equation}\label{4.14}
\begin{aligned}
(\cL_{a\star}u_\d,u_\d)_{L_2(\mathds{R}^d)} =& \sum\limits_{\substack{m,q\in\mathds{Z}_+^d \\ |m|+|q|\leqslant 2N }} (-1)^{|q|} \d^{|m|+|q|} \p^{m+q}a(0)
\frac{U_m \overline{U_q}}{m!q!}
\\
&+ \d^{2N+1} \int\limits_{Q_1(0)\times Q_1(0)} A_{2N+1}(x-y,\d) U(x)\overline{U(y)}\,dxdy.
\end{aligned}
\end{equation}

Estimate (\ref{4.13}) yields immediately that
\begin{equation*}
\left|\d^{2N+1} \int\limits_{Q_1(0)\times Q_1(0)} A_{2N+1}(x-y,\d) U(x)\overline{U(y)}\,dxdy\right| \leqslant C\d^{2N+1} \|U\|_{L_2(Q_1(0))}^2,
\end{equation*}
where $C$ is some constant independent of $\d$ and $U$.

For now on we specify the choice of the function $U$. Namely, we assume that it is a polynomial of degree at most $N$, i.e.,
\begin{equation}\label{4.17}
U(x)=\sum\limits_{\substack{m\in\mathds{Z}^d_+ \\ |m|\leqslant N}} c_m x^m.
\end{equation}
Then we have
\begin{equation}\label{4.16}
U_n=\sum\limits_{\substack{m\in\mathds{Z}^d_+ \\ |m|\leqslant N}} c_m \int\limits_{Q_1(0)} x^{n+m}\,dx.
\end{equation}
A matrix of size $M(N)\times M(N)$ with entries $\int\limits_{Q_1(0)} x^{n+m}\,dx$, is the Gram matrix of linearly independent functions $\{x^n\}$, $n\in\mathds{Z}_+^d$, $|n|\leqslant N$, and hence, this matrix is non-degenerate. Then it follows from (\ref{4.16}) that the coefficients $c_m$ are expressed as linear combinations of $U_n$, $n\in\mathds{Z}_+^d$, $|n|\leqslant N$. Therefore, each polynomial (\ref{4.17}) can be equivalently characterized be means of the coefficients $U_n$, $n\in\mathds{Z}_+^d$, $|n|\leqslant N$. In particular, this implies uniform estimates
\begin{align}\label{4.19}
&\tilde c^{-1}\|U\|_{L_2(Q_1(0))}^2\leqslant \sum\limits_{\substack{n\in\mathds{Z}^d_+ \\ |n|\leqslant N}} |U_n|^2 \leqslant \tilde c
\sum\limits_{\substack{n\in\mathds{Z}^d_+ \\ |n|\leqslant N}} |U_n|^2,
\\
&\label{4.20}
\sum\limits_{\substack{n\in\mathds{Z}^d_+ \\ N+1\leqslant |n|\leqslant 2N}} |U_n|^2 \leqslant C\|U\|_{L_2(Q_1(0))}^2,
\end{align}
where $\tilde c$ and $C$ are constants independent of $U$.

We rewrite the first term in the right hand side of (\ref{4.14}) as
\begin{equation}\label{4.21}
\begin{aligned}
\sum\limits_{\substack{m,q\in\mathds{Z}_+^d \\ |m|+|q|\leqslant 2N }} (-1)^{|q|} \d^{|m|+|q|} \p^{m+q}a(0)
\frac{U_m \overline{U_q}}{m!q!}
= &
\sum\limits_{\substack{m,q\in\mathds{Z}_+^d \\ |m|,|q|\leqslant N}} (-1)^{|q|} \d^{|m|+|q|} \p^{m+q}a(0)
\frac{U_m \overline{U_q}}{m!q!}
\\
+ &\sum\limits_{\substack{m,q\in\mathds{Z}_+^d, \, |m|+|q|\leqslant 2N \\ |m|\geqslant N+1\,\text{or}\, |q|\geqslant N+1}} (-1)^{|q|} \d^{|m|+|q|} \p^{m+q}a(0)\frac{U_m \overline{U_q}}{m!q!}.
\end{aligned}
\end{equation}
By (\ref{4.20}) we can estimate the second term in the right hand side of the above identity as follows:
\begin{equation}\label{4.22}
\left|\sum\limits_{\substack{m,q\in\mathds{Z}_+^d,\, |m|+|q|\leqslant 2N \\ |m|\geqslant N+1\,\text{or}\, |q|\geqslant N+1}} (-1)^{|q|} \d^{|m|+|q|} \p^{m+q}a(0)
\frac{U_m \overline{U_q}}{m!q!}
\right|\leqslant C\d^{N+1} \|U\|_{L_2(Q_1(0))}^2,
\end{equation}
where $C$ is a constant independent of $\delta$ and $U$. In view of the definition of the form $\fa_N$ in (\ref{2.12}), the first term in the right hand side of (\ref{4.21}) can be expressed as
\begin{equation*}
\sum\limits_{\substack{m,q\in\mathds{Z}_+^d \\ |m|,|q|\leqslant N }} (-1)^{|q|} \d^{|m|+|q|} \p^{m+q}a(0)
\frac{U_m \overline{U_q}}{m!q!}
=\fa_N[\ru_\d],\qquad \ru_\d:=\left(\d^{|m|} \frac{U_m}{m!}\right)_{m\in\mathds{Z}_+^d,\, |m|\leqslant N}.
\end{equation*}
It follows from (\ref{4.19}) that
\begin{equation}\label{4.24}
\|\ru_\d\|_{\mathds{C}^{M(N)}}^2\geqslant \tilde c_0\d^{2N} \|U\|_{L_2(Q_1(0))}^2,
\end{equation}
where $\tilde c_0$ is a positive constant independent of $\delta$ and $U$.

Since by the assumptions of the theorem the form $\fa_N$ is negative definite on the subspace $S$, there exists a constant $\tilde c_1>0$ independent of $\d$ and $\ru_\d$ such that
\begin{equation*}
\fa_N[\ru_\d]\leqslant -\tilde c_1 \|\ru_\d\|_{\mathds{C}^{M(N)}}^2
\end{equation*}
for each $\ru_\d\in S$. Hence, for each polynomial (\ref{4.17}), for which the corresponding vector $\ru_\d$ belongs to $S$, by (\ref{4.24}) we obtain:
\begin{equation*}
\fa_N[\ru_\d] \leqslant -\tilde c_1 \tilde c_0 \d^{2N} \|U\|_{L_2(Q_1(0))}^2.
\end{equation*}
The above inequality and (\ref{4.22}) allow us to estimate the form $(\cL_{a\star} u_\d,u_\d)_{L_2(\mathds{R}^d)}$ from above for sufficiently small $\delta$ as follows:
\begin{equation}\label{4.25}
(\cL_{a\star} u_\d,u_\d)_{L_2(\mathds{R}^d)} \leqslant \d^{2(N+d)} (-\tilde c_0 \tilde c_1+\d \tilde c)\|U\|_{L_2(Q_1(0))}^2\leqslant
-\frac{\tilde c_0 \tilde c_1}{2}\d^{2(N+d)} \|U\|_{L_2(Q_1(0))}^2.
\end{equation}

We proceed to estimating the contribution of the potential $V$ to the form of the operator $\cL$. Namely, we have:
\begin{align*}
(\cL_{V} u_\d, u_\d)_{L_2(\mathds{R}^d)}- V_{\min} \|u_\d\|_{L_2(\mathds{R}^d)}^2 =& \int\limits_{\mathds{R}^d} (V(x)-V_{\min}) |u_\d(x)|^2\,dx
\\
=&\d^d \int\limits_{Q_1(0)} \big(V(x_0+\d x)-V_{\min}\big)|U(x)|^2\,dx
\\
=&\d^d \sum\limits_{\substack{m,q\in\mathds{Z}_+^d \\ |m|,|q|\leqslant 2N}}
c_n \overline{c_m} \int\limits_{Q_1(0)} \big(V(x_0+\d x)-V_{\min}\big) x^{n+m}\,dx.
\end{align*}
Hence, by definition (\ref{2.13}) of the function $h_N(\d)$ and inequality (\ref{4.19}) we find:
\begin{equation*}
(\cL_{V} u_\d, u_\d)_{L_2(\mathds{R}^d)}- V_{\min} \|u_\d\|_{L_2(\mathds{R}^d)}^2\leqslant C\d^d h_N(\d) \|U\|_{L_2(Q_1(0))}^2,
\end{equation*}
where $C$ is some constant independent of $\d$ and $U$. This estimate and (\ref{4.25}) yield:
\begin{align*}
(\cL u_\d,u_\d)_{L_2(\mathds{R}^d)} - V_{\min}\|u_\d\|_{L_2(\mathds{R}^d)}^2 \leqslant \d^{2(N+d)} \left(-\frac{\tilde c_0
\tilde c_1}{2} + C \frac{h_N(\d)}{\d^{2N+d}} \right) \|U\|_{L_2(\mathds{R}^d)}^2.
\end{align*}
Applying condition (\ref{2.14}), we finally see that for sufficiently small $\d$ the estimate
\begin{equation*}
(\cL u_\d,u_\d)_{L_2(\mathds{R}^d)} \leqslant -\frac{\tilde c_0 \tilde c_1}{3}\d^{2(N+d)}\|U\|_{L_2(\mathds{R}^d)}^2
\end{equation*}
holds true for each polynomial $U$ defined by formula (\ref{4.17}),
for which the corresponding vector $\ru_\d$ belongs to $S$.
Since the dimension of the space of such polynomials coincides with that of the subspace $S$, by the minimax principle we conclude that the operator $\cL$ possesses at least $\dim S$ discrete eigenvalues below $\mu_0$. The proof of Theorem~\ref{thDS4} is complete.

\subsection{Proof of Theorem~\ref{thDS6}}

We introduce a subspace $S$ in $\mathds{C}^{M(N)}$ consisting of vectors $\z=(\z_n)_{ n\in\mathds{Z}_+^d,\, |n|\leqslant N}$ such that $\z_n=0$ as $n\notin \mathds{I}$. It is obvious that the dimension of the space $S$ coincides with $\#\mathds{I}$.

Let us show that the conditions of this theorem imply the assumptions of Theorem~\ref{thDS4} with the introduced subspaces $S$. We begin with studying the form $\fa_N$. We rewrite definition (\ref{2.12}) of this form as
\begin{equation}\label{5.1}
\fa_N[\z]=\sum\limits_{n\in\mathds{I}}
(-1)^{|n|} \p^{2n} a(0) |\z_n|^2 + \sum\limits_{n,m\in\mathds{I}, \, n\ne m} (-1)^{|n|}
\p^{n+m} a(0)\z_m\overline{\z_n}.
\end{equation}
Inequalities (\ref{2.19}) then yield
\begin{equation}\label{5.2}
\sum\limits_{n\in\mathds{I}}
(-1)^{|n|} \p^{2n} a(0) |\z_n|^2 = -\sum\limits_{n\in\mathds{I}}
|\p^{2n} a(0)| |\z_n|^2.
\end{equation}
Employing condition (\ref{2.20}) and assuming that (\ref{2.21_bis}) holds, we estimate the second term in (\ref{5.1}) as follows:
\begin{align*}
\Bigg|\sum\limits_{n,m\in\mathds{I}, \, n\ne m}(-1)^{|n|}
\p^{n+m} a(0)\z_m\overline{\z_n}\Bigg| &\leqslant \sum\limits_{n,m\in\mathds{I}, \, n\ne m} \b_{n,m}  \sqrt{|\p^{2n}a(0)|} \sqrt{|\p^{2m}a(0)|} |\z_n||\z_m|
\\
&\leqslant \frac12\sum\limits_{n\in\mathds{I}}\sum\limits_{\substack{m\in\mathds{I}, \\ m\ne n}}\beta_{n,m}|\p^{2n}a(0)| |\z_n|^2\ + \
 \frac12\sum\limits_{m\in\mathds{I}}\sum\limits_{\substack{n\in\mathds{I}, \\ n\ne m}}\beta_{n,m}|\p^{2m}a(0)| |\z_m|^2
\\
& \leqslant \b_1\sum\limits_{n\in\mathds{I}}|\p^{2n}a(0)| |\z_n|^2.
\end{align*}
Combining this estimate with  identities (\ref{5.1}), (\ref{5.2})  we obtain the following estimate for the form $\fa_N$:
\begin{equation*}
\fa_N[\z]\leqslant -(1-\b_1)\sum\limits_{n\in\mathds{I}}  |\p^{2n}a(0)| |\z_n|^2,
\end{equation*}
and, since  $\b_1<1$, we conclude that the form $\fa_N$ is negative definite.

Recalling the definition of $h_N(\d)$  in (\ref{2.13}), it is straightforward to show that estimate (\ref{2.22}) implies the following estimate:
\begin{equation*}
|h_N(\d)|\leqslant C\d^\a,
\end{equation*}
where $C$ is some constant independent of $\d$. Since $2N<\a-d$ by the assumption of $N$, condition (\ref{2.14}) is fulfilled, and therefore
Theorem  \ref{thDS4} applies.

If (\ref{2.21}) holds, the second term on the right-hand side of (\ref{5.1}) can be estimated as follows:
\begin{align*}
\Bigg|\sum\limits_{n,m\in\mathds{I}, \, n\ne m} (-1)^{|n|}
\p^{n+m} a(0)\z_m\overline{\z_n}\Bigg| \leqslant  & \sum\limits_{n,m\in\mathds{I}, \, n\ne m} \b_{n,m}  \sqrt{|\p^{2n}a(0)|} \sqrt{|\p^{2m}a(0)|} |\z_n||\z_m|
\\
\leqslant&\Big(\sum\limits_{\substack{n,m\in\mathds{I}, \\ n\ne m}} \b^2_{n,m}\Big)^\frac12
\Big(\sum\limits_{\substack{n,m\in\mathds{I}, \\ n\ne m}}|\p^{2n}a(0)|\,|\p^{2m}a(0)|\, |\z_n|^2|\z_m|^2\Big)^\frac12\\
=&\b_2^\frac{1}{2}\bigg(\Big(\sum\limits_{n\in\mathds{I}}|\p^{2n}a(0)|\, |\z_n|^2\Big)^2-\sum\limits_{n\in\mathds{I}}|\p^{2n}a(0)|^2 |\z_n|^4\bigg)^\frac12
\end{align*}
Therefore,
\begin{align*}
\fa_N[\z]\leqslant& -\sum\limits_{n\in\mathds{I}}
\big|\p^{2n} a(0)\big| |\z_n|^2+\b_2^\frac{1}{2} \bigg(\Big(\sum\limits_{n\in\mathds{I}}|\p^{2n}a(0)|\, |\z_n|^2\Big)^2-\sum\limits_{n\in\mathds{I}}|\p^{2n}a(0)|^2 |\z_n|^4\bigg)^\frac12
\\[4mm]
=&\frac{\b_2 \bigg(\Big(\sum\limits_{n\in\mathds{I}}|\p^{2n}a(0)|\, |\z_n|^2\Big)^2-\sum\limits_{n\in\mathds{I}}|\p^{2n}a(0)|^2 |\z_n|^4\bigg)-\Big(\sum\limits_{n\in\mathds{I}} \big|\p^{2n} a(0)\big| |\z_n|^2\Big)^2}
{ \b_2^\frac12  \bigg(\Big(\sum\limits_{n\in\mathds{I}}|\p^{2n}a(0)|\, |\z_n|^2\Big)^2-\sum\limits_{n\in\mathds{I}}|\p^{2n}a(0)|^2 |\z_n|^4\bigg)^\frac12 +\sum\limits_{n\in\mathds{I}}|\p^{2n}a(0)| |\z_n|^2}
\\[4mm]
\leqslant&-\big(1-\b_2\big)\sum\limits_{n\in\mathds{I}}|\p^{2n}a(0)| |\z_n|^2-\b_2
\Big(\sum\limits_{n\in\mathds{I}}|\p^{2n}a(0)|^2 |\z_n|^4\Big)\Big(\sum\limits_{n\in\mathds{I}}|\p^{2n}a(0)|\, |\z_n|^2\Big)^{-1}
\end{align*}
Considering the inequality
$$
\Big(\sum\limits_{n\in\mathds{I}}|\p^{2n}a(0)| |\z_n|^2\Big)^2\leqslant (\#\mathds{I})\sum\limits_{n\in\mathds{I}}|\p^{2n}a(0)|^2 |\z_n|^4,
$$
we derive from the latter estimate the following upper bound:
$$
\fa_N[\z]\leqslant- \big( (1-\b_2)+(\#\mathds{I})\big.^{-\frac12}\b_2\big)\Big(\sum\limits_{n\in\mathds{I}}|\p^{2n}a(0)|^2 |\z_n|^4
\Big)\Big(\sum\limits_{n\in\mathds{I}}|\p^{2n}a(0)|\, |\z_n|^2\Big)^{-1}.
$$
Therefore, the form  $\fa_N[\z]$ is negative definite if  $1-\b_2+(\#\mathds{I})^{-\frac12}\b_2>0$ or, equivalently,
$$
\b_2 <\frac{(\#\mathds{I})^{\frac12}}{(\#\mathds{I})^{\frac12}-1}.
$$
This completes the proof of Theorem \ref{thDS6}.

\subsection{Proof of Theorem \ref{th_suf-infin}}

It suffices to check that under the assumptions of Theorem \ref{th_suf-infin} there exists an infinite subset
$\tilde{\mathds{I}}\subset{\mathds{I}}$  such that condition (\ref{2.21_bis}) holds for $\tilde{\mathds{I}}$. Indeed, letting
$\tilde{\mathds{I}}_N=\{n\in \tilde{\mathds{I}}\,:\,|n|\leqslant N\}$, $N\in\mathds Z_+$, and assuming  that (\ref{2.21_bis}) holds
for $\tilde{\mathds{I}}$, by Theorem \ref{thDS6}
we obtain that there exist at least $\#\tilde{\mathds{I}}_N$ points of the discrete spectrum of $\mathcal{L}$ below $\mu_0$. Since $N$ is an arbitrary
number from $\mathds{N}$, and $\#\tilde{\mathds{I}}_N$ tends to infinity as $N\to\infty$,  the desired statement follows.

It remains to construct a subset $\tilde{\mathds{I}}$ that satisfies the aforementioned  conditions.
 \begin{lemma}\label{l_order}
   There exists an infinite sequence $n^1, n^2,\ldots, n^j,\ldots$ with $n^j\in\mathds{I}$ such that
   $n^{j+1}_k\geqslant n^j_k$ for all $k=1,\ldots, d$ and all $j\in\mathds{Z}_+$, and $n^{j+1}\ne n^j$.
 \end{lemma}

We choose the indices $j_\ell$ in such a way that at least for one $k\in\{1,\ldots,d\}$  the inequality
$$
n^{j_{\ell+1}}_k\geqslant 2|n^{j_\ell}|
$$
 holds.
 Denote  $|n|_\infty=\max\limits_k n_k$.
Then for  each  $m\in \mathds Z_+$ and $\ell\in\mathds Z_+$, $m<\ell$ we have
\begin{equation*}
\frac{(2n^{j_\ell})!\,(2n^{j_m})!} {\big((n^{j_\ell}+n^{j_m})!\big)^2}=
\prod\limits_{k=1}^d \frac{(2n_k^{j_\ell})!\,(2n_k^{j_m})!} {\big((n_k^{j_\ell}+n_k^{j_m})!\big)^2}\geqslant
\frac{(2n_{k_0}^{j_\ell})!\,(2n_{k_0}^{j_m})!} {\big((n_{k_0}^{j_\ell}+n_{k_0}^{j_m})!\big)^2}
\geqslant
\Big(\frac32\Big)^{\frac12|n^{j_\ell}|_\infty};
\end{equation*}
here the index $k_0$ is such that $|n^{j_\ell}|_\infty=n_{k_0}^{j_\ell}$.  In view of conditions (\ref{con2}) and (\ref{con3}), this estimate yields
the following inequalities:
$$
\frac{(|\p^{2n^{j_\ell}}a(0)|\,|\p^{2n^{j_m}}a(0)|)^\frac12}{|\p^{n^{j_\ell}+n^{j_m}}a(0)|}
\geqslant \frac{c_1\big((2n^{j_\ell})!\,(2n^{j_m})!\big)^\frac\gamma2}{c_2\big((n^{j_\ell}+n^{j_m})!\big)^\gamma}
\geqslant \frac{c_1}{c_2}\,\Big(\frac32\Big)^{\frac\gamma4|n^{j_\ell}|_\infty}
$$
Choosing $n^{j_\ell}$ in such a way that
$$
\frac{c_1}{c_2}\,\Big(\frac32\Big)^{\frac\gamma4|n^{j_\ell}|_\infty} \geqslant 2^\ell\quad\text{for all}\quad\ell\in\mathds{Z}_+,
$$
we obtain the desired subset $\tilde{\mathds{I}}$ and complete the proof.

\section{Upper bound for the number of discrete eigenvalues}

In this section we prove Theorem~\ref{thDS3} establishing in this way an upper bound for the number of the discrete eigenvalues under the threshold of the essential spectrum. The proof of this theorem follows the main ideas of the Birman-Schwinger principle,
see e.g. \cite[Thm. XIII.10]{RS-4}, but
with appropriate modifications needed for our operator $\cL$.

We begin with introducing an auxiliary operator
\begin{equation*}
\cL^{(-)}:= - (\cL_{a_-\star} + \cL_{V_-}),
\end{equation*}
where $a_-$ is defined by the identity $\cF[a_-] = \hat{a}_-$. According to the definition of the functions $\hat{a}_-$ and $V_-$ and by identities (\ref{3.2}) we conclude immediately that both operators $\cL_{a_-\star}$ and $\cL_{V_-}$ are non-positive. Hence, the operator $\cL^{(-)}$ is non-negative.

We denote by $E_n$ and $E_n^{(-)}$ respectively the discrete eigenvalues of the operators $\cL$ and $\cL^{(-)}$ below $V_{\min}$ taken counting their multiplicities. By $N_0$ and $N^{(-)}$ we denote respectively the total number of the eigenvalues $E_n$ and $E_n^{(-)}$, that is,
$$
N_0 = \# \{n:\, E_n < \mu_0\}, \qquad N^{(-)} = \# \{ n:\, E_n^{(-)} < \mu_0\}.
$$
Then expression (\ref{3.20}) for the form of the operators $\cL$, a similar expression for the form of the operator $\cL^{(-)}$ and the minimax principle imply that
\begin{equation}\label{4.27}
N_0 \leqslant N^{(-)}.
\end{equation}
Hence, it is sufficient to find an upper bound for $N^{(-)}$.

We observe that if some $E<\mu_0\leqslant0$ is an eigenvalue of the operator $\cL^{(-)}$ and a corresponding eigenfunction
 $\psi$ solving the equation
$$
\big( E- \cL^{(-)} \big)\psi=0,
$$
then the function $\vp:=V_-^{\frac{1}{2}}\psi$ is a solution of the equation
\begin{equation}\label{2.4.1}
\vp=-V_-^{\frac{1}{2}} \big( \cL_{a_-\star} + E\big)^{-1} V_-^{\frac{1}{2}} \vp.
\end{equation}
Here the function $V_-^{\frac{1}{2}}$ is well-defined and non-negative since the function $V_-$ is non-negative by its definition. Equation (\ref{2.4.1}) also means that $1$ is an eigenvalue of the operator $-V_-^{\frac{1}{2}} \big( \cL_{a_-\star} + E\big)^{-1} V_-^{\frac{1}{2}}$ if $E$ is an eigenvalue of the operator $\cL^{(-)}$.

Since by the assumption of the theorem  we have
$$
\min \hat{a}_-=\inf \hat{a}=a_{\min}\geqslant V_{\min}=\mu_0,
$$
then for $E<\mu_0\leqslant 0$ the inverse operator $(\cL_{a_-\star}+E)^{-1}$ is well-defined and bounded in $L_2(\mathds{R}^d)$. It can be easily found by means of formulae (\ref{3.2}):
\begin{equation}\label{4.28}
\begin{aligned}
(\cL_{a_-\star}+E)^{-1}=& \left(\frac{1}{(2\pi)^\frac{d}{2}}
\cF\right)^{-1} (\cL_{\hat{a}_-}+E)^{-1} \left(\frac{1}{(2\pi)^\frac{d}{2}}
\cF\right)
\\
=& \left(\frac{1}{(2\pi)^\frac{d}{2}}
\cF\right)^{-1} (\hat{a}_- + E)^{-1} \left(\frac{1}{(2\pi)^\frac{d}{2}}
\cF\right).
\end{aligned}
\end{equation}
Then, owing to a simple identity
\begin{equation*}
\frac{1}{\hat{a}_- + E}=\frac{1}{E} (1+b_E),\qquad b_E:=-\frac{\hat{a}_-}{\hat{a}_-+E},
\end{equation*}
and (\ref{3.2}), we can rewrite formula (\ref{4.28}) as
\begin{equation}\label{2.4.2}
( \cL_{a_-\star} + E )^{-1} = \frac{1}{E}\big(\cI+\cL_{\hat{b}_E\star}\big),\qquad
\hat{b}_E:=\cF[b_E].
\end{equation}

We observe that the function $\hat{b}_E(\xi)$ is strictly positive for all $\xi\in\mathds{R}^d$ and it increases monotonically in $E$. This implies immediately that the operator $\cL_{\hat{b}_E\star}$ is monotonically increasing in $E$ in the sense of quadratic forms.

Substituting (\ref{2.4.2}) into (\ref{2.4.1}), we obtain
\begin{equation*}
\vp=-\frac{V_-}{E}\vp -\frac{1}{E} V_-^{\frac{1}{2}} \cL_{\hat{b}_E\star} V_-^{\frac{1}{2}} \vp.
\end{equation*}
Taking into consideration that $E+V_->0$, we denote $\phi:=(-(E+V_-))^{\frac{1}{2}}\vp$ and we
rewrite the above equation as
\begin{equation*}
\cQ_E \phi=\phi,\qquad \cQ_E:=
\big( - (E + V_-) \big)^{-\frac{1}{2}} V_-^{\frac{1}{2}} \cL_{\hat{b}_E\star} V_-^{\frac{1}{2}} \big( - (E + V_-) \big)^{-\frac{1}{2}} \phi.
\end{equation*}
Hence, if $E<\mu_0$ is an eigenvalue of the operator $\cL^{(-)}$, then $1$ is an eigenvalue of operator $\cQ_E$.

We observe that $\cQ_E$ is an integral operator:
\begin{equation*}
(\cQ_E u)(x)=\int\limits_{\mathds{R}^d} Q_E(x,y) u(y)\,dy,
\end{equation*}
where
\begin{equation}\label{4.30}
Q_E(x,y):=\big(-(E +V_-(x))\big)^{\frac{1}{2}}V_-^{\frac{1}{2}}(x) \hat{b}_E(x-y) V_-^{\frac{1}{2}}(y)\big(-(E + V_-(y))\big)^{\frac{1}{2}}.
\end{equation}

By their definitions, both the functions $V_-$ and $\hat{b}_E$ vanish at infinity. This ensures that the operator $\cL_{\hat{b}_E\star}$ is compact in $L_2(\mathds{R}^d)$ and therefore, the same is true for the operator $\cQ_E$. And since the operator $\cL_{\hat{b}_E\star}$ is monotonically increasing in $E$ in the sense of quadratic forms, we obtain the same property also for $\cQ_E$. These two properties of the operator $\cQ_E$ yield that first,
the spectrum of the operator $\cQ_E$ consists of discrete eigenvalues $\l_m(E)$ and a possible point of continuous spectrum at zero, and second, these eigenvalues $\l_m(E)$ increase as
$E \to \mu_0 -0$. In view of the latter property and the aforementioned relation between the eigenvalues of the operator $\cL^{(-)}$ and of $\cQ_E$, if $E_n^{(-)}$ is an eigenvalue of the operator $\cL^{(-)}$, then $\l_m(E_n^{(-)})=1$ for some $m$ and $\l_m(E)>1$ as $E<E_n^{(-)}$. Hence, in order to count the total number of the eigenvalues $E_n^{(-)}$, it is sufficient to count the total number of the eigenvalues $\l_m(E)$ passing through $1$ as $E$ goes to $\mu_0$ from below.
In view of an obvious inequality
\begin{equation*}
\sum\limits_{m:\, \l_m(E)\geqslant 1} 1\leqslant \sum\limits_{m:\, \l_m(E)\geqslant 1} \l_m(E)
\end{equation*}
by (\ref{4.30}) we then get:
\begin{align*}
N^{(-)} =&\lim\limits_{E\to \mu_0-}\sum\limits_{m:\, \l_m(E)\geqslant 1} 1 \leqslant \lim\limits_{E\to \mu_0-}\sum\limits_{m:\, \l_m(E)\geqslant 1} \l_m(E)=\lim\limits_{E\to \mu_0-}\operatorname{Tr} \cQ_{E}
\\
=& \lim\limits_{E\to \mu_0-} \int\limits_{\mathds{R}^d} Q_E (x,x)\, dx
= \lim\limits_{E\to \mu_0-} \hat{b}_E(0) \int\limits_{\mathds{R}^d} \frac{V_- (x)}{-(E + V_-(x))} \, dx = I_a I_V.
\end{align*}
In view of inequality (\ref{4.27}), this completes the proof.

\section{Infinite discrete spectrum}

In this section we discuss the situations when the operator $\cL$ possesses infinitely many points of the discrete sprectrum, namely, we prove Theorem~\ref{thDS2}.

We first assume that inequalities (\ref{2.10}) are satisfied.
Since $V(x)\equiv V_{\min}$ on $Q_r(x_0)$, in view of (\ref{3.20}) for each
infinitely differentiable function $u$ compactly supported in
$Q_r(x_0)$
we have
\begin{equation*}
\fl[u]-V_{\min}\|u\|_{L_2(\mathds{R}^d)}^2 = (\cL_{a\star}u,u)_{L_2(\mathds{R}^d)} =(\hat{a}\hat{u},\hat{u})_{L_2(\mathds{R}^d)},\qquad \hat{u}:=\frac{1}{(2\pi)^\frac{d}{2}}
\cF[u].
\end{equation*}
Inequalities (\ref{2.10}) imply that $\hat{a}\leqslant 0$ and $\hat{a}$ is a non-trivial function. Therefore, this function is non-zero and negative on a set of positive measure; we denote this set by $\Om$. Since the function $u$ is compactly supported, its Fourier transform $\hat{u}$ is analytic in $\xi$. Hence, it is non-zero on $\Om$ and we get
\begin{equation*}
\fl[u]-V_{\min}\|u\|_{L_2(\mathds{R}^d)}^2 < 0\qquad\text{for each}\quad u\in C_0^\infty(Q_r(x_0)).
\end{equation*}
Since $C_0^\infty(Q_r(x_0))$
 is an infinite dimensional space, by the minimax principle the above inequality implies that the operator $\cL$ possesses infinitely many eigenvalues below the point $\mu_0$.

We proceed to the proving the second part of the theorem. Due to (\ref{4.35}), for an arbitrary function $u\in C_0^\infty(\overline{Q_r(x_0)})$ continued by zero outside $\overline{Q_r(x_0)}$ we have:
\begin{equation}\label{4.2}
\fl[u]-V_{\min}\|u\|_{L_2(\mathds{R}^d)}^2 = (\cL_{a\star}u,u)_{L_2(\mathds{R}^d)}\leqslant \sum\limits_{n\in\mathds{Z}^d} a_n|U_n|^2.
\end{equation}
By our assumptions, all Fourier coefficients satisfy $a_n\leqslant 0$ and there exists an infinite subsequence of these coefficients, for which the latter inequality is strict. We denote such subsequence by $n'$ and then by (\ref{4.2}) we get:
\begin{equation*}
\fl[u]-V_{\min}\|u\|_{L_2(\mathds{R}^d)}^2\leqslant \sum\limits_{n'} a_{n'}|U_{n'}|^2<0
\end{equation*}
for $u$ such that at least one of its coefficients $U_{n'}$ is non-zero. It is clear that the space of such functions is infinite-dimensional and by the minimax principle we conclude on the existence of countably many eigenvalues below $\mu_0$. This completes the proof.

\section*{Acknowledgments}

D.I.B. was partially supported by  the Czech Science Foundation within the project 22-18739S.

\end{document}